\documentclass{article}
\usepackage{amsthm}

\usepackage[english]{babel}
\usepackage[utf8x]{inputenc}
\usepackage[T1]{fontenc}
\usepackage{graphicx}

\usepackage[version=0.96]{pgf}
\usepackage{tikz}
\usetikzlibrary{arrows,shapes,decorations,automata,backgrounds,petri,cd}

\usepackage{algorithm2e}

\usepackage{amsmath,amssymb,amsfonts,bbm}

\usepackage{float}

\usepackage[justification=centering]{caption}

\usepackage{mathtools}

\newtheorem{thm}{Theorem}[section]
\newtheorem{prop}[thm]{Proposition}

\newtheorem{lemme}[thm]{Lemma}

\newtheorem{props}[thm]{Properties}
\newtheorem{define}[thm]{Definition}
\newtheorem{def-prop}[thm]{Definition-Proposition}

\newtheorem{ex}[thm]{Example}
\newtheorem{rem}[thm]{Remark}
\newtheorem{hypo}[thm]{Hypothesis}

\newcommand\A{\mathcal A}

\newcommand\ane{\mathcal O}

\newcommand\N{\mathbb N}
\newcommand\Z{\mathbb Z}
\newcommand\Q{\mathbb Q}
\newcommand\R{\mathbb R}
\newcommand\C{\mathbb C}
\newcommand\D{\mathcal D}

\newcommand\abs[1]{\left|#1\right|}
\newcommand\norm[1]{\left\|#1\right\|}

\newcommand\codim{\operatorname{codim}}

\newcommand\K{\mathbb K}
\newcommand\defi[1]{\textbf{#1}}

\newcommand\rk{\operatorname{rank}}
\newcommand\im{\operatorname{Im}}

\newcommand\ab{\operatorname{ab}}
\newcommand\cod{\operatorname{cod}}
\newcommand\interior{\operatorname{int}}

\renewcommand\S{\mathbb{S}}
\newcommand\T{\mathbb{T}}

\newcommand{\restr}[1]{|_{#1}}

\usepackage{hyperref}

\usepackage{pdfpages}

\author{Paul MERCAT}

\title{Geometrical representation of subshifts for primitive substitutions}

\begin{document}

\maketitle

\begin{abstract}
	For any primitive substitution whose Perron eigenvalue is a Pisot unit, we construct a domain exchange that is measurably conjugate to the subshift.
	And we give a condition for the subshift to be a finite extension of a torus translation.
	For the particular case of weakly irreducible Pisot substitutions, we show that the subshift is either a finite extension of a torus translation or its eigenvalues are roots of unity.
	And we provide an algorithm to compute eigenvalues of the subshift associated with any primitive pseudo-unimodular substitution.
\end{abstract}

\tableofcontents

\section{Introduction and main results}

In the seminal paper~\cite{Rauzy}, G. Rauzy constructed a geometrical representation of the subshift associated with some particular substitution. He constructed a compact subset of $\R^2$ which is called now Rauzy fractal, and that tiles the plane and gives a measurable conjugacy between the subshift and a translation on the torus $\T^2$. It was generalized later by many people as Arnoux-Ito, see~\cite{AI}.

For irreducible Pisot unit substitutions, it is conjectured that Rauzy fractals give a measurable conjugacy between the subshift and a translation on a torus.
What is known is that it gives a finite extension of a torus translation:

\begin{thm}[Host, unpublished] \label{thm:host}
	Let $\sigma$ be an irreducible Pisot unimodular substitution over an alphabet of $d+1$ letters. Then the uniquely ergodic subshift $(\Omega_\sigma, S)$ is a finite extension of a translation on the torus $\T^d$. 
\end{thm}

Recently, F. Durand and S. Petite gave a very interesting proof of this result in~\cite{durand-petite}.
Their starting point is to construct a proper substitution whose subshift is conjugate to the subshift of the first substitution (see Theorem~\ref{thm:dp}). But this construction doesn't preserve irreducibility.
So they have to deal with reducible substitutions.
Moreover, primitive reducible substitutions naturally arise from some dynamical systems (see for example the $9$-letter substitution in~\cite{ABB}, coming from an interval exchange transformation).

We use the strategy of the proof of F. Durand and S. Petite to extend Theorem~\ref{thm:host} to a large class of pseudo-unimodular substitutions, i.e. substitutions whose product of all non-zero eigenvalues of the incidence matrix equals $\pm1$.

\begin{thm} \label{thm:ef}
	Let $\sigma$ be a proper primitive pseudo-unimodular substitution.
	Assume that the Perron eigenvalue of the incidence matrix is a Pisot number $\beta$ of degree $d+1$.
	And assume that every generalized eigenvector for every other eigenvalue
	of modulus $\geq 1$, has sum zero.
	Then, the subshift $(\Omega_\sigma, S)$ is a finite extension of a minimal translation on the torus $\T^d$.
\end{thm}

Moreover, we show that the Pisot hypothesis is necessary (see Proposition~\ref{prop:pisot}).
This theorem, together with Theorem~\ref{thm:dp}, permits to check easily that many non-proper substitutions have a subshift which is a finite extension of a torus translation.
However, we don't know if the reciprocal of Theorem~\ref{thm:ef} true.
But in the particular case of weakly irreducible Pisot substitutions (i.e. each eigenvalue of the primitive incidence matrix is a Pisot unit or conjugate, zero, or a root of unity),
we have the following alternative.

\begin{thm} \label{thm:wip}
	Let $\sigma$ be a weakly irreducible Pisot substitution.
	Then one of the following is true
	\begin{itemize}
		\item eigenvalues of the subshift $(\Omega_\sigma, S)$ are roots of unity,
		\item the subshift $(\Omega_\sigma, S)$ is a finite extension of a minimal translation of the torus $\T^d$, where $d+1$ is the degree of the Pisot number.
	\end{itemize}
	Moreover, there is an algorithm to decide in which case we are.
\end{thm}

Notice that for unimodular substitutions, the first point implies that the subshift is weakly mixing by Lemma~\ref{lem:unimod}.
We give a geometrical representation of the subshift for any primitive substitutions whose Perron eigenvalue is a Pisot unit.

\begin{thm} \label{thm:de}
	Let $\sigma$ be a primitive substitution such that the Perron eigenvalue of its incidence matrix is a unit Pisot number of degree $d+1$.
	Then, the uniquely ergodic subshift $(\Omega_\sigma, S)$ is measurably isomorphic to a domain exchange $(R,E, \lambda)$, with $R \subseteq \R^d$.
\end{thm}

In the particular case of irreducible substitutions, this result is due to F. Durand and S. Petite.
Notice that this result is similar to Theorem~6 in~\cite{BV}, but in this paper, we give a proof, and Theorem~\ref{thm:dp} of F. Durand and S. Petite permits us to avoid the strong coincidence hypothesis.
The set $R$ is a Rauzy fractal, but potentially for another substitution.
This result, together with Theorem~\ref{thm:ef} gives a generalization of the main theorem in~\cite{durand-petite}.

We also give a way to compute eigenvalues of the subshift associated with any primitive pseudo-unimodular substitution.

\begin{thm} \label{thm:vp}
	Let $\sigma$ be a primitive proper pseudo-unimodular substitution over an alphabet $A$.
	Then $e^{2 i \pi\alpha}$ is an eigenvalue of $(\Omega_\sigma, S)$ if and only if
	there exists a row vector $w \in \Z^A$ such that for every generalized eigenvector $v$ of $M_\sigma$ for eigenvalues of modulus $\geq 1$, we have
	\begin{itemize}
		\item $w v = \alpha$ if $v$ has sum $1$,
		\item $w v = 0$ if $v$ has sum $0$.
	\end{itemize}
\end{thm}

In this theorem, it is enough to check the condition for any choice of bases of generalized eigenspaces for eigenvalues of modulus $\geq 1$.
And thanks to the proprification algorithm of F. Durand and S. Petite (see Subsection~\ref{ss:prop}), this theorem permits to completely describe and compute the set of eigenvalues of the subshift for any primitive pseudo-unimodular substitution.
We provide an algorithm and an implementation doing this computation (see section~\ref{sec:cvp}).
Notice that a different way to compute eigenvalues is given in~\cite{FMN}.

Note that Theorem~\ref{thm:vp} implies that eigenvalues of $(\Omega_\sigma, S)$ are in the form $e^{2i\pi \alpha}$ with $\alpha$ in a free $\Z$-module of finite rank in $\Q(\beta)$, where $\beta > 1$ is the Perron eigenvalue of $M_\sigma$.

The hypothesis that $\sigma$ is proper and pseudo-unimodular in Theorem~\ref{thm:vp} are needed only for the direct implication.
Moreover this hypothesis can be lightened.
All we need is the fact that if $e^{2 i \pi \alpha}$ is an eigenvalue of the subshift, then $\alpha (1, \dots, 1) M^n \xrightarrow[n \to \infty]{} 0$ mod $\Z^A$.
It is the case if there is no non-trivial coboundary and if the initials period is $1$ (see~\cite{host} for more details, and see~\cite{mosse}).
It is in particular the case if a power of the substitution is left-proper.
And in Theorem~\ref{thm:ef} the hypothesis that $\sigma$ is proper can also be replaced with the hypothesis that a power of $\sigma$ is left-proper, since it also implies the strong coincidence property.

\subsection{Organization of the paper}

We start in section~\ref{sec:def} by definitions and notations.
Then, in section~\ref{sec:gr}, we introduce the notion of generalized Rauzy fractal.
It permits to generalize the notion of Rauzy fractal to reducible substitutions, with various possible choice of projection.
It will permit to get translations on torus as factor of the subshift and also to get domain exchanges, but with different projections.
The section~\ref{sec:r} focus on the particular choice of projection giving usual Rauzy fractals, for which we get many nice properties.
In section~\ref{sec:de} we prove Theorem~\ref{thm:de}, by constructing usual Rauzy fractals that permits to get domain exchanges.
In section~\ref{sec:vp} we prove Theorem~\ref{thm:vp}.
Then in section~\ref{sec:cvp} we explicit an algorithm to compute the eigenvalues.
In section~\ref{sec:ef}, we prove Theorem~\ref{thm:ef} and Theorem~\ref{thm:wip}.
We finish with section~\ref{sec:ex} by giving examples.

\section{Definitions and notations} \label{sec:def}

This section aims to give all the definitions and notations that will be used in the paper.

\subsection{Algebraic numbers}
An \defi{algebraic number} $\beta$ is a root of a polynomial with rational coefficients.
The smallest unitary polynomial $P$ with rational coefficients such that $P(\beta) = 0$ is called \defi{minimal polynomial}.
The \defi{degree} of $\beta$ is the degree of its minimal polynomial.
Two different algebraic numbers are \defi{conjugate} if they have the same minimal polynomial.
An algebraic number $\beta$ is an \defi{algebraic integer} if coefficients of its minimal polynomial are in $\Z$.
An algebraic number $\beta$ is a \defi{unit} if it is an algebraic integer such that the constant term of its minimal polynomial is $\pm 1$. This is equivalent to saying that $\beta$ and $1/\beta$ are algebraic integers.
A \defi{Pisot number} is an algebraic integer $\beta > 1$ whose conjugates $\gamma$ satisfy $\abs{\gamma} <1$.
 
\subsection{Words and worms}
 An \defi{alphabet} is a finite set.
 If $A$ is an alphabet, then we denote by $A^*$ the set of \defi{finite words} over $A$.
 We denote by $\abs{u}$ the length of a word $u$.
 An \defi{occurrence} of a word $w$ in a word $u$ is the length $\abs{p}$ of a word $p$ such that $u = pws$, where $s \in A^*$ is a word.
 We denote by $\abs{u}_w$ the number of occurrences of $w$ in $u$.
 The \defi{abelianization} of a finite word $u \in A^*$ is the vector $\ab(u) = (\abs{u}_a)_{a \in A}$.
 For every letter $a \in A$, we denote $e_a = \ab(a)$. The family $(e_a)_{a \in A}$ is the canonical basis of $\R^A$.
 
 The set of \defi{bi-infinite words} over $A$ is $A^\Z$.
 \defi{Infinite words} over $A$ are elements of $A^\N$.
 For a (bi-)infinite word $u$, and for every $n \in \N$, we use the standard notation $u_{[0,n)} = u_0u_1...u_{n-1}$.
 For $n < 0$, we use the convention $\ab(u_{[0,n)}) = -\ab(u_{[-n,0)})$.

The usual metric on $A^\Z$ is defined for $u \neq v$ by
\[
	d(u,v) = 2^{-n}, \text{ where } n = \max\{ k \in \N \mid u_{[-k,k]} = v_{[-k,k]} \}.
\]
For this metric, $A^\Z$ is compact.
A \defi{subshift} $(\Omega, S)$ is a compact subset $\Omega \subseteq A^\Z$ which is invariant under the \defi{shift map}:
$
	S : \begin{array}{rcl}
			A^\Z &\to& A^\Z \\
			(u_i)_{i \in \Z} &\mapsto& (u_{i+1})_{i \in \Z}
		\end{array}.
$
The \defi{orbit} of a bi-infinite word $u \in A^\Z$ is
$\ane(u) = \{S^n u \mid n \in \Z\}$.
A subshift $(\Omega, S)$ is said to be \defi{minimal} if every orbit is dense in $\Omega$,
and \defi{aperiodic} if every orbit is infinite.

 We define the \defi{worm} associated to a bi-infinite word $u \in A^\Z$ as
\[
	W(u) = \{ \ab(u_{[0,n)}) \mid n \in \Z \}.
\]
We also define
\[
	W_a(u) = \{ x \in W(u) \mid x + e_a \in W(u) \}.
\]
The notion of worm can also be defined for infinite words in an obvious way (see~\cite{pytheas}).

\begin{props}
	\begin{itemize}
		\item For every $u \in \A^\Z$, we have $W(Su) = W(u) - \ab(u_0)$.
		\item We have
			\[
					W(u) = \bigcup_{a \in A} W_a(u) = \left( \bigcup_{a \in A} W_a(u) + e_a \right) \cup \{0\},
			\]
			and these unions are disjoint.
	\end{itemize}
\end{props}

\subsection{Matrices and subspaces}

    We denote by $I_n \in M_n(\N)$, or just $I$ when there is no ambiguity, the identity matrix.
	A matrix is said to be \defi{irreducible} if its characteristic polynomial is irreducible.
	Let $M \in M_n(\N)$ be a matrix.
	We say that $M$ is \defi{primitive} if there exists $n \geq 1$ such that every coefficient of $M^n$ is strictly positive.
	We say that $M$ is \defi{pseudo-unimodular} if the product of all its non-zero eigenvalues, counting multiplicities, equals $\pm1$.
	In particular, unimodular matrices are pseudo-unimodular.
	
	We use the following well-known theorem.
	\begin{thm}[Perron-Frobenius]
		If $M \in M_n(\N)$ is primitive, then $M$ has a simple real eigenvalue equal to the spectral radius of $M$.
		Moreover, the corresponding eigenvector can be chosen with strictly positive components.
	\end{thm}
	
	We call this maximal eigenvalue the \defi{Perron eigenvalue} of $M$, and we call the associated eigenvector a \defi{Perron eigenvector}.
	
	We say that $v \in \C^n$ is a \defi{generalized eigenvector} of $M$ for an eigenvalue $\beta$ if $v$ is a non-zero vector in the \defi{generalized eigenspace} $\ker((M - \beta I)^k)$ where $k \geq 1$ is the algebraic multiplicity of $\beta$.
	
	We extend the notion of projector to linear maps that are not endomorphisms. We say that a linear map $V : \R^n \to \R^d$ is a \defi{projection} along a subspace $F$ of $\R^n$ if $\ker(V) = F$ and $d+\dim(F) = n$. Such a map is onto.
	
	We have the following lemma.
	\begin{lemme} \label{lem:proj}
		Let $V : \R^n \to \R^d$ be a projection along $F$, and let $M : \R^n \to \R^n$ be a linear map such that $M(F)=F$.
		Then there exists a unique linear map $N : \R^d \to \R^d$ such that $NV = VM$, and we have $\det(N) = \det(M')$,
		where $M' : \R^n/F \to \R^n/F$ is the quotient map.
	\end{lemme}
	
	\begin{proof}
		The map $V$ gives an isomorphism $V' : \R^n/F \to \R^d$, and the map $M$ gives a map $M' : \R^n/F \to \R^n/F$.
		Then, we can define $N$ by $N = V'M'(V')^{-1}$, and it satisfies $NV = VM$ and $\det(N) = \det(M')$.
		The unicity comes from the fact that $V$ is onto: if $N$ and $N'$ are two such maps, then $(N-N')V = 0$, so $N = N'$.
	\end{proof}
	
	We say that a subspace is \defi{rational} if it admits a basis with coefficients in $\Q$.
	We say that a vector $v$ has a \defi{totally irrational direction} if the coefficients of $v$ are linearly independent over $\Q$. A projection is \defi{totally irrational} if it is a projection along a vector with a totally irrational direction.
	
\subsection{Substitutions} \label{sec:subst}

	We say that a morphism $\sigma : A^* \to A^*$ is \defi{non-erasing} if for every $a \in A$, $\abs{\sigma(a)} \geq 1$.
	A \defi{substitution} over an alphabet $A$ is a non-erasing morphism of $A^*$.
	The \defi{incidence matrix} of a substitution $\sigma$ is the matrix $M_\sigma = (\abs{\sigma(a)}_b)_{(b,a) \in A^2}$.
	For any finite word $u \in A^*$, we have the relation $\ab(\sigma(u)) = M_\sigma \ab(u)$.
	We say that a substitution is \defi{primitive}, \defi{irreducible}, or any property that has a meaning for a matrix, if its incidence matrix has the corresponding property.
	
	The \defi{subshift} of a primitive substitution $\sigma$ is the dynamical system $(\Omega_\sigma, S)$, where $\Omega_\sigma$ is the smallest non-empty compact subset of $A^\Z$ invariant under the substitution and by the shift map. We denote it $\Omega$ when there is no ambiguity.
	It can be shown that for every primitive substitution $\sigma$, the subshift $(\Omega_\sigma, S)$ is minimal and uniquely ergodic (see Subsection~V.2 and Theorem~V.13 in ~\cite{queffelec}).
	We say that the substitution $\sigma$ is \defi{aperiodic} if every orbit in the subshift is infinite.
	Notice that if a substitution is primitive and pseudo-unimodular, then it is aperiodic since the Perron eigenvalue is irrational.
	
	For every finite words $v,w \in A^*$, we denote by $[v \cdot w]$ the \defi{cylinder} of $\Omega$
	\[
		[v \cdot w] = \{ u \in \Omega \mid u_{[0,\abs{w})} = w \text{ and } u_{[-\abs{v}, 0)} = v \}.
	\]
	And we denote $[w] = \{ u \in \Omega \mid u_{[0,\abs{w})} = w\}$.
	
	A \defi{fixed point} of a substitution $\sigma$ is a bi-infinite word $u \in A^\Z$ such that $\sigma(u) = u$.
	A \defi{periodic point} of $\sigma$ is a bi-infinite word $u \in A^\Z$ such that there exists $n \geq 1$ such that $u$ is a fixed point of $\sigma^n$.
	We say that a fixed point or a periodic point is \defi{admissible} if it is an element of the subshift $\Omega_\sigma$.
	Every primitive substitution has an admissible periodic point.
	
	We say that a substitution $\sigma$ is \defi{left-proper} (respectively \defi{right-proper}) if there exists a letter $a_0 \in A$ such that for every $a \in A$, $\sigma(a)$ starts (respectively ends) with letter $a_0$.
	The substitution is \defi{proper} if it is left-proper and right-proper.
	
	The following theorem is due to F. Durand and S. Petite (see Corollary~9 in~\cite{durand-petite}).

\begin{thm}[Durand-Petite] \label{thm:dp}
	Let $\sigma$ be a primitive substitution.
	Then there exists a proper primitive substitution $\xi$ such that
	\begin{itemize}
		\item $(\Omega_\sigma, S)$ is conjugate to $(\Omega_\xi, S)$,
		\item there exists $l \geq 1$ such that the substitution matrices $M_\sigma^l$ and $M_\xi$ have the same eigenvalues, except perhaps $0$ and $1$.
	\end{itemize}
	Moreover, the proof is effective.
\end{thm}
	
	We call \defi{proprification algorithm} an algorithm that inputs a primitive substitution $\sigma$ and that outputs a proper substitution $\xi$ as in this theorem. We say that we \defi{proprify} a substitution $\sigma$ if we apply to it such algorithm, and the output substitution is called \defi{proprification} of $\sigma$.
    See Subsection~\ref{ss:prop} for more details about the proprification algorithm of F. Durand and S. Petite.
	
	We say that a substitution is \defi{weakly irreducible Pisot} if it is primitive, the Perron eigenvalue $\beta$ is a unit Pisot number, and every other eigenvalue of its incidence matrix is either a conjugate of $\beta$, a root of unity, or zero.
	Note that the class of primitive pseudo-unimodular substitutions is strictly larger than this (e.g. Example~\ref{ex:fmn1}, and irreducible Salem substitutions).

\subsection{Prefix-suffix automaton and Dumont-Thomas numeration} \label{ss:dt}

	Let $\sigma$ be a substitution over an alphabet $A$.
	The \defi{prefix-suffix automaton} of $\sigma$ is an automaton whose states are the set $A$, and whose transitions
	are $a \xrightarrow{p,s} b$ for every letters $a,b \in A$ and words $p,s \in A^*$ such that $\sigma(a) = pbs$.
	In all this article, we denote by $a \xrightarrow{p,s} b$ if and only if $\sigma(a) = pbs$, if there is no ambiguity on what is the substitution $\sigma$.
	
	The \defi{abelianized prefix automaton} is the same automaton where we replace transitions
	$a \xrightarrow{p,s} b$ by $a \xrightarrow{t} b$, where $t = \ab(p)$.
	
	For the subshift $\Omega_\sigma$ and for every letter $a \in A$, we have the relation
	\[
		[a] = \bigcup_{b \xrightarrow{p,s} a} S^{\abs{p}} \sigma([b]).
	\]
	And we have a similar relation for worms: for every $u \in \Omega$, we have
	\[
		W_a(\sigma(u)) = \bigcup_{b \xrightarrow{t} a} M_\sigma W_b(u) + t.
	\]
	
	Any word $u \in \Omega$ can be written uniquely in the form
	\[
		u = \sigma^n(v_n) \sigma^{n-1}(p_{n-1}) ... \sigma(p_1) p_0 \cdot b s_0 \sigma(s_1) ... \sigma^{n-1}(s_{n-1}) \sigma^n(w_n),
	\]
	for $v_n$ a left-infinite word, $w_n$ a right-infinite word, and such that we have a path
	$
	 \xrightarrow{p_{n-1}, s_{n-1}} \dots \xrightarrow{p_0, s_0} b
	$
	in the prefix-suffix automaton.
	We call \defi{sequence of prefixes} the sequence $(p_n)_{n \in \N} = (p_n(u))_{n \in \N}$ associated to $u$.
	The \defi{sequence of abelianized prefixes} of $u \in \Omega$ is defined by $t_n(u) = \ab(p_n(u))$.
	
	Hence, to any word $u \in \Omega$, we associate a unique left-infinite path in the prefix-suffix automaton or in the abelianized prefix automaton (see Proposition~3.2 in~\cite{CS}).
	Notice that such a path can be considered as a path in a Bratelli diagram of $\Omega$.
	
\subsection{Eigenvalues of a subshift}
 
 We denote by $\S^1$ the set of complex numbers of modulus one.
 We say that $\eta \in \S^1$ is an \defi{eigenvalue} of a subshift $(\Omega, S)$ if there exists a continuous function $f : \Omega \to \S^1$ called \defi{eigenfunction} such that $f \circ S = \eta f$.
 Notice that for primitive substitutions, if we allow the eigenfunctions to be only measurable rather than continuous, and the image to be $\C$ rather than $\S^1$, it doesn't give more eigenvalues (see~\cite{host}, Theorem~1.4).
 
 We say that a subshift $(\Omega, S)$ is an \defi{extension} of a translation on a torus $\T^d = \R^d/\Z^d$, if
 there exists a continuous map $f : \Omega \to \T^d$ and $\alpha \in \T^d$ such that
 $f \circ S = f + \alpha$.
 We say moreover that this extension is \defi{finite} if the cardinality of $f^{-1}(x)$ is finite for almost every $x \in \T^d$, for the \defi{Lebesgue measure} that we denote $\lambda$.
 
 Notice that if $1$, $\alpha_1$, ..., $\alpha_d \in \R$ are linearly independent over $\Q$, then the translation by $\alpha = (\alpha_1, ..., \alpha_d)$ on the torus $\T^d$ is minimal and uniquely ergodic.
 It implies that an eigenfunction $f$ is necessarily almost everywhere constant-to-one, but with a constant that can be infinite if the extension is not finite.
 
 We say that a subshift $(\Omega, S)$ is \defi{weakly mixing} if its only eigenvalue is $1$, and if this eigenvalue $1$ is simple. Notice that if $\sigma$ is a primitive substitution, then eigenvalues of $(\Omega_\sigma, S)$ are simple since it is uniquely ergodic.

\subsection{Domain exchange}

We call \defi{domain exchange} a subset $R \subseteq \R^d$, with a map $E : R \to R$ almost everywhere defined for the Lebesgue measure $\lambda$ such that there exists a finite number of subsets $R_a$, $a \in A$, such that
\begin{itemize}
	\item $R = \bigcup_{a \in A} R_a$, and the union is Lebesgue disjoint,
	\item each $R_a$ is the closure of its interior,
	\item the boundary of each $R_a$ has zero Lebesgue measure,
	\item for every $a \in A$, $E \restr{int(R_a)}$ is a translation,
	\item $\lambda(R) = \lambda(E(R))$.
\end{itemize}

We say that a map $f : R \to \R^d$, where $R \subset \R^d$, is a \defi{translation by pieces}
if there exists a finite measurable partition $R = \bigcup_{i \in I} R_i$ such that for every $i \in I$ the restriction $f \restr{E_i}$ is a translation.

Notice that the map $E$ associated with a domain exchange is a translation by pieces.
And note that a translation by pieces is finite-to-one.

\section{Generalized Rauzy fractals} \label{sec:gr}

In this section, we generalize the usual notion of Rauzy fractal.
As we will see, the construction depends on the choice of a projection map. For subshifts associated to irreducible substitutions, the choice of the projection is obvious, but for primitive substitutions several choices can be made.
One choice gives a domain exchange, and another one permits to get a translation on a torus as a factor.

\begin{prop} \label{prop-rauzy-fractal}
	Let $(\Omega, S)$ be a minimal aperiodic subshift over an alphabet $A$, and let $V : \R^A \to \R^d$ be a linear map.
	Assume that there exists $u \in \Omega$ such that $V W(u)$ is bounded.
	Then, the map
	\[
		\begin{array}{rcl}
			\phi : \ane(u) &\to& \R^d \\
					S^n u & \mapsto & V \ab(u_{[0,n)})
		\end{array}
	\]
	can be extended by continuity to the whole subshift $\Omega$.
\end{prop}

\begin{proof}
    This proposition is a generalization of Lemma~8.2.5 in~\cite{AM}, with a more general projection map $V$, and with bi-infinite words rather than right-infinite words.
    But the same proof works.
\end{proof}

We call the image $\phi(\Omega)$ a \defi{Rauzy fractal} of $\Omega$.
We denote this map $\phi_{u,V,\Omega}$, and we will omit $u$, $V$, or $\Omega$ when there is no ambiguity.

\begin{rem}
	With this definition, a Rauzy fractal is always compact.
	It is possible to give a more general definition that allows unbounded Rauzy fractal as in~\cite{melo}.
\end{rem}

The following proposition gives properties of the map $\phi$.

\begin{prop} \label{props-rauzy-fractal}
	Under the hypothesis of Proposition~\ref{prop-rauzy-fractal}, we have
	\begin{itemize}
		\item The Rauzy fractal $R = \phi(\Omega)$ is the closure of $VW(u)$.
		\item For every $v \in \Omega$, $\phi(Sv) = \phi(v) + V\ab(v_0)$.
		\item For every $v \in \Omega$, $\phi_v$ is well-defined and $\phi_u - \phi_v$ is constant.
		\item If $v$ and $w$ are two bi-infinite words of $\Omega$ with the same left-infinite or right-infinite part, then $\phi(v) = \phi(w)$.
	\end{itemize}
\end{prop}

\begin{proof}
	By continuity of $\phi$, $R$ is the closure of $\phi(\ane(u))$.
	And by construction, $\phi(\ane(u)) = VW(u)$.
	Thus $R$ is the closure of $VW(u)$.

	By construction, we have for every $n \in \Z$, $\phi(S^{n+1}u) - \phi(S^nu) = V\ab(u_n)$.
	Since $\Omega$ is minimal the orbit of $u$ is dense in $\Omega$, and since $\phi$ is continuous, we get that
	for every $v \in \Omega$, $\phi(Sv) = \phi(v) + V\ab(v_0)$.

	Let $v \in \Omega$. Then, the set
	\[
		VW(v) = \{V \ab(v_{[0,n)}) \mid n \in \Z\} = \{\phi_u(S^n v) - \phi_u(v) \mid n \in \Z\}
	\]
	is bounded, so $\phi_v$ is well-defined.
	And for every $n \in \Z$, we have $\phi_u(S^nv) - \phi_v(S^nv) = \phi_u(v) + V\ab(v_{[0,n)}) - V\ab(v_{[0,n)}) = \phi_u(v)$.
	By density of the orbit of $v$ and by continuity, we get that $\phi_u - \phi_v$ is constant to $\phi_u(v)$.

	If $u$ and $v$ are two elements of $\Omega$ having their right-infinite parts in common, then the proof of Proposition~\ref{prop-rauzy-fractal} shows that $\phi(u) = \phi(v)$.
	If it is the left-infinite parts that $u$ and $v$ have in common, then we come back to the previous case by symmetry, looking at the mirror of the words.
\end{proof}

The following proposition permits us to show that the Rauzy fractal is well-defined for the subshift of a substitution, as soon as the projection and the incidence matrix are compatible.

\begin{lemme} \label{lem:cv}
	Let $\sigma$ be a primitive and aperiodic substitution over an alphabet $A$, and let $u$ be an admissible fixed point of $\sigma$.
	If $V : \R^A \to \R^d$ is a linear map such that $\sum_{n \in \N} \norm{V M_\sigma^n}$ converges,
	then the hypothesis of Proposition~\ref{prop-rauzy-fractal} is satisfied and $\phi_{u, V, \Omega_\sigma} : \Omega_\sigma \to \R^d$ is well-defined.
	Moreover, for every $v \in \Omega$ we have the equality
	\[
		\phi_{u, V, \Omega}(v) = \sum_{n=0}^\infty V M^n t_n(v),
	\]
	where $t_n(v) = \ab(p_n(v))$ is defined in subsection~\ref{ss:dt}.
\end{lemme}

\begin{proof}
	The subshift $(\Omega_\sigma,S)$ is minimal since $\sigma$ is primitive.
	We have $\phi_{V}(\ane(u)) = \{ V \ab(u_{[0,n)}) \mid n \in \Z\}$. The positive part is described by
	\[
		\{ V \ab(u_{[0,n)}) \mid n \in \N\} 
							= \{ \sum_{n=0}^N VM^n t_n \mid u_0 \xrightarrow{t_N} a_{N} \dots a_1 \xrightarrow{t_0} a_0 , N \in \N \}
	\]
	Since $t_n$ are in a finite set (abelianizations of prefixes of $\sigma(a)$, $a \in A$), and since $\sum_{n \in \N} \norm{VM^n}$ converges, we get that the set is bounded.
	The negative part can be described in the same way and is also bounded.
	Thus, $VW(v)$ is bounded, so $\phi_{V} : \Omega_\sigma \to \R^d$ is a well-defined continuous map.
	
	To prove the last equality, note that the sum $f = \sum_{n=0}^\infty V M^n t_n$ defines a continuous map $f: \Omega \to \R^d$ since every $t_n : \Omega \to \R^A$ is continuous and since the series is normally convergent.
	Hence, it is enough to check the equality on the dense subset $\{S^k u \mid k \in \N\}$.
	Let $k \in \N$.
	There exists $N \in \N$ such that $t_n(S^ku) = 0$ for every $n \geq N$.
	Then, we have
	\begin{eqnarray*}
		f(S^k u) &=& \sum_{n=0}^{N-1} V M^n t_n(S^k u) \\
					&=& V \ab(\sigma^{N-1}(p_{N-1}(S^k u)) \dots \sigma(p_1(S^k u))p_0(S^k u)) \\
					&=& V \ab(u_{[0, k)}) \\
					&=& \phi(S^k u).
	\end{eqnarray*}
\end{proof}

\section{Usual Rauzy fractal} \label{sec:r}

The previous section defined Rauzy fractal for general subshifts and for various choices of projections.
In this section, we focus on subshifts associated with primitive substitutions whose Perron eigenvalue of the incidence matrix is a Pisot unit, and we consider a particular choice of projection that permits to have many nice properties.
More precisely, we assume the following.

\begin{hypo} \label{hyp:rauzy}
	\text{ }
	\begin{itemize}
		\item $\sigma$ is a primitive substitution over an alphabet $A$ such that the Perron eigenvalue of $M_\sigma$ is a unit Pisot number $\beta$ of degree $d+1$,
		\item $u \in \Omega_\sigma$ is an admissible fixed point of $\sigma$,
		\item $V : \R^A \to \R^d$ is a projection along $\ker((M-\beta I)P(M))$, where $P \in \Z[X]$ is such that the characteristic polynomial of $M_\sigma$ has the form $\pi_\beta P$, where $\pi_\beta$ is the minimal polynomial of $\beta$. In other words, $V$ is a projection along every generalized eigenspace except for the conjugates of modulus less than $1$ of the Perron eigenvalue $\beta$,
		\item $\phi = \phi_{u, V, \sigma}$, $R = \phi(\Omega_\sigma)$ and for every $a \in A$, $R_a = \phi([a])$.
	\end{itemize}
\end{hypo}

\begin{define}
	Under Hypothesis~\ref{hyp:rauzy}, we say that $R$ is an \defi{usual Rauzy fractal} of $\sigma$.
\end{define}

Such usual Rauzy fractals have a lot of nice properties.

\begin{props} \label{props:usual:rauzy}
	Under Hypothesis~\ref{hyp:rauzy}, we have the following properties
	\begin{itemize}
		\item There exists a unique invertible endomorphism $N$ of $\R^d$ such that $NV = V M_\sigma$, and $\abs{\det(N)} = \frac{1}{\beta}$,
		\item the union $R_a = \bigcup_{b \xrightarrow{t} a} N R_b + Vt$ is disjoint in Lebesgue measure,
		\item $V$ restricted to $W(v)$ is one-to-one, for every $v \in A^\Z$,
		\item each $R_a$ is the closure of its interior,
		\item each $R_a$ has a boundary of zero Lebesgue measure.
	\end{itemize}
\end{props}

In the following of this section, we prove these properties.

The map $N$ is given by Lemma~\ref{lem:proj}.
The determinant of $N$ is equal to the determinant of the quotient map $M: \R^A/F \to \R^A/F$, where $F = \ker((M-\beta I)P(M))$.
The eigenvalues of this quotient map are all the roots of $\pi_\beta$ but $\beta$. 
The hypothesis that the Perron eigenvalue is a Pisot unit gives us that the product of the roots of $\pi_\beta$ is $\pm 1$, thus we get
$\det(N) = \frac{\pm 1}{\beta}$.

Now, we give several lemmas that permits to prove the other properties.

\begin{lemme} \label{lem:gIFS}
	We assume Hypothesis~\ref{hyp:rauzy}.
	Then, the pieces $R_a$, $a \in A$, of the Rauzy fractal are the smallest non-empty compact solutions of the equations
	\[
		N^{-1} R_a = \bigcup R_b + \D_{a,b}, \quad a \in A
	\]
	where $\D_{a,b} = \{N^{-1} Vt \mid b \xrightarrow{t} a \}$.
\end{lemme}

\begin{proof}	
	Since $u$ is a fixed point, we have the union
	\[
		W_a(u) = \bigcup_{b \xrightarrow{t} a} M W_b(u) + t.
	\]
	Then, applying $V$ to both sides and using $NV = VM$, we get
	\[
		R_a = \bigcup_{b \xrightarrow{t} a} N R_b + Vt.
	\]
	Now, assume that $R_a'$, $a \in A$, are non-empty compact sets satisfying such equations.
	Since we have $\norm{N} < 1$, iterating such equations gives for every $a \in A$
	\[
	    \{ \sum_{n=0}^\infty N^n V t_n \mid ... \xrightarrow{t_n} ... \xrightarrow{t_0} a \} \subset R_a'.
	\]
	Thus, by Lemma~\ref{lem:cv} we have $R_a \subseteq R_a'$. So $R_a$, $a \in A$, are indeed the smallest non-empty compact subsets satisfying the equations.
\end{proof}

If we iterate the equations of this lemma, we get
\[
	N^{-n} R_a = \bigcup R_b + \D_{a,b}^n, \quad a \in A,
\]
where $\D_{a,b}^n = \{ \sum_{i=0}^{n-1} N^{i-n}Vt_i \mid b \xrightarrow{t_{n-1}} ... \xrightarrow{t_0} a \}$.

\begin{lemme} \label{lem:sep}
	There exists $\epsilon > 0$ such that for every $a,b \in A$ and every $n \in \N$, the set $\D_{a,b}^n$ is $\epsilon$-separated: $\forall x \neq y \in \D_{a,b}^n$, $\norm{x-y} > \epsilon$.
\end{lemme}

In the following, we need some notations.
The projection map $V : \R^A \to \R^d$ can be factorized: $V = V_\beta V_P$, where
$V_P : \R^A \to \ker(\pi_\beta(M))$ is the projection onto $\ker(\pi_\beta(M))$ along $\ker(P(M))$,
	and where $V_\beta : \ker(\pi_\beta(M)) \to \R^d$ is a projection along the Perron eigenspace of $M$.
	The projection $V_\beta$ is totally irrational, since it is a projection along the Perron eigenspace of the endomorphism $M \restr{\ker(\pi_\beta(M))}$ whose characteristic polynomial $\pi_\beta$ is irreducible.
	
\begin{proof}
	As $V_P$ is a projection along a rational subspace, $\Lambda = V_P \Z^A$ is a lattice of $\ker(\pi_\beta(M))$.
	And we have $M \Lambda \subseteq \Lambda$ since $\ker(\pi_\beta(M))$ is invariant under $M$.
	We have $\det(M\restr{\ker(\pi_\beta(M))}) = \pm 1$ since $\beta$ is assumed to be an algebraic unit.
	Thus, we have $M^{-1} \Lambda = \Lambda$.
	
	Now let us consider the set
	\[
		\mathcal{T}_{a,b}^n = \{ \sum_{i=0}^{n-1} M^{i-n}V_P t_i \mid b \xrightarrow{t_n-1} ... \xrightarrow{t_0} a \}.
	\]
	It is a subset of $\Lambda$ since every $t_i$ is in $\Z^A$.
	Moreover, since the $t_i$ are in a finite set, it stays at a bounded distance $D > 0$ of a hyperplane $\mathcal{P}$ of $\ker(\pi_\beta(M))$ which is the orthogonal complement of a left Perron eigenvector of $M\restr{\ker(\pi_\beta(M))}$.
	Let
	\[
		\mathcal{T}_D = \{ x \in \Lambda \mid d(x, \mathcal{P}) \leq D \}.
	\]
	For every $\alpha > 0$, the set $\mathcal{T}_{2D} \cap V_\beta^{-1} B(0, \alpha)$ is finite, since $\ker(V_\beta) \oplus \mathcal{P} = \im(V_P)$.
	Thus, as $V_\beta$ is totally irrational, there exists $\epsilon > 0$ such that $B(0, \epsilon) \cap V_\beta \mathcal{T}_{2D}$ has cardinality one.
	Then, for every $x, y \in V_\beta \mathcal{T}_D$ such that $\norm{x-y} \leq \epsilon$, we have $x-y \in V_\beta \mathcal{T}_{2D} \cap B(0, \epsilon)$ by triangular inequality, so $x=y$.
	In other words, the set $V_\beta \mathcal{T}_D$ is $\epsilon$-separated.
	
	Since we have for every $a,b \in A$ and every $n \in \N$, $\D_{a,b}^n = V_\beta \mathcal{T}_{a,b}^n \subset V_\beta \mathcal{T}_D$, we get the result.
\end{proof}

Such subsets $V_\beta \mathcal{T}_D$ of $\R^d$ are sometimes called cut-and-project sets.

\begin{lemme} \label{lem:inj}
	Under Hypothesis~\ref{hyp:rauzy}, the projection $V$ is one-to-one on $W(v)$, for any bi-infinite word $v \in A^\Z$.
\end{lemme}

\begin{proof}
	Let $V_{\hat \beta} : \R^A \to \ker(M-\beta I)$ be the projection on $\ker(M-\beta I)$ along $\ker(Q(M))$ where $(X-\beta)Q(X)$ is the characteristic polynomial of $M$.
	Let us show that $V_{\hat \beta}$ is one-to-one on $W(v)$.
	We have $V_{\hat \beta} M = \beta V_{\hat \beta}$, so the matrix of $V_{\hat \beta}$ is a left Perron eigenvector of $M$ for any choice of basis of $\ker(M-\beta I)$.
	Thus, we can choose a basis of $\ker(M-\beta I)$ such that it has strictly positive coordinates.
	Now, if we take two distinct elements of $W(v)$, their difference is the abelianization of a non-empty word, so it is a non-negative and non-zero vector of $\Z^A$.
	Thus, its image by $V_{\hat \beta}$ is strictly positive, and $V_{\hat \beta}$ is one-to-one on $W(v)$.
	As, we have $V_{\hat \beta} = V_{\hat \beta} V_P$, it proves that $V_P$ is also one-to-one on $W(v)$.
	Then, the total irrationality of $V_\beta$ and the fact that $V_P W(v)$ is rational give the result.
\end{proof}

\begin{lemme} \label{lem:meas}
	Under Hypothesis~\ref{hyp:rauzy}, the Lebesgue measure of $R_a$ is non-zero for every $a \in A$.
\end{lemme}

\begin{proof}
	The proof is similar to the proof of Proposition~2.8 in~\cite{SW}.
	Thanks to Lemma~\ref{lem:sep}, one can choose $\epsilon > 0$ such that for every $n \in \N$, the set $\D_a^n = \bigcup_{b \in A} \D_{a,b}^n$ is $2\epsilon$-separated.
	Thanks to Lemma~\ref{lem:inj}, the cardinality of the set $\D_{a, u_0}^n$ is $\abs{\sigma^n(u_0)}_a = e_a^t M^n e_{u_0}$, where $u_0$ is the first letter of the fixed point $u$.
	Hence,
	\[
		\lambda(\bigcup_{x \in \D_a^n} B(x, \epsilon)) \geq \sum_{x \in D_{a, u_0}^n} \lambda(B(x,\epsilon)) = e_a^t M^n e_{u_0} \lambda(B(0,\epsilon)).
	\]
	Let $D$ be large enough such that for every $b \in A$, $N \bigcup_{x \in \D_b^n} B(x, D) \subseteq B(0,D)$.
	Then the sequence of sets $N^n \bigcup_{x \in \D_{a}^n} B(x, D)$ decreases, and its intersection is $R_a$,
	thus
	\[
		\lambda(R_a) = \lim_{n \to \infty} \lambda \left( N^n \bigcup_{x \in \D_{a}^n} B(x, D) \right) \geq \liminf_{n \to \infty} \lambda(N^n \bigcup_{x \in \D_{a, u_0}^n} B(x, \epsilon)).
	\]
	This limit is strictly greater than zero since $\frac{1}{\beta^n} M^n$ converges to the matrix in the canonical basis of the projector $V_{\hat \beta}$ defined in the proof of Lemma~\ref{lem:inj}, and since we have $e_a^t V_{\hat \beta} e_{u_0} > 0$.	
\end{proof}

\begin{lemme} \label{lem:nei}
	Under Hypothesis~\ref{hyp:rauzy}, for every $a \in A$, $R_a$ has a non-empty interior and is the closure of its interior.
\end{lemme}

Note that this lemma has similarities with Lemma~8.3.4 in~\cite{AM} but it is not equivalent:
it has different hypotheses, different conclusion and the tools used are not the same even if in both cases the idea is to use the self-similarity of the objects.
To show this lemma, we use the following theorem due to V.F. Sirvent and Y. Wang, see Theorem~3.1 in~\cite{SW}.

\begin{thm}[Sirvent-Wang] \label{thm:sw}
	Let $(X_1, ..., X_J)$ be the attractor of a strongly connected graph-directed IFS
	\[
		A(X_i) = \bigcup_{j = 1}^J (X_j + \D_{ij}), i = 1, ..., J.
	\]
	Assume that there exists $\epsilon > 0$ such that the sets $\D_{i,j}^m$ are $\epsilon$-separated for all $i$, $j$, $m$,
	and assume that $X_1$ has a positive Lebesgue measure.
	Then every $X_i$ has a non-empty interior and is the closure of its interior.
\end{thm}

\begin{proof}[Proof of Lemma~\ref{lem:nei}]
	Let us show that hypotheses of Theorem~\ref{thm:sw} are fulfilled.
	The sets $R_a$, $a \in A$, are the attractor of an equation of this form, with $A = N^{-1}$ and $\D_{ij} = \{ N^{-1} Vt \mid j \xrightarrow{t} i \}$ thanks to Lemma~\ref{lem:gIFS}.
	The sets $\D_{a,b}^n$ are $\epsilon$-separated by Lemma~\ref{lem:sep}.
	And the sets $R_a$ have non-zero Lebesgue measure by Lemma~\ref{lem:meas}.
	Therefore, we can apply the theorem and it gives the result.
\end{proof}

It remains to show that the union
\[
	R_a = \bigcup_{b \xrightarrow{t} a} N R_b + Vt,
\]
is disjoint in measure.
We follow a classical argument due to Host (see~\cite{AI}).
We have the inequality
\[
	\lambda(R_a) \leq \sum_{b \xrightarrow{t} a} \lambda(NR_b) = \frac{1}{\beta} \sum_{b \xrightarrow{t} a} \lambda(R_b).
\]
Let $X = (\lambda(R_a))_{a \in A}$. We get the inequality
$
	X \leq \frac{1}{\beta} M X
$
since the matrix of the prefix-suffix automaton is $M$.
Now we use the following lemma.

\begin{lemme}[Perron-Frobenius] \label{lem-Perron}
	Let $M$ be a primitive positive matrix, with maximal eigenvalue $\lambda$.
	Suppose that $v$ is a positive vector such that $M v \geq \lambda v$.
	Then the inequality is an equality, and $v$ is an eigenvector with respect to $\lambda$.
\end{lemme}

\begin{proof}
	See~\cite{AI}, Lemma~11.
\end{proof}

We deduce from this lemma that the inequality $X \leq \frac{1}{\beta} MX$ is an equality, thus the union $R_a = \bigcup_{b \xrightarrow{t} a} N R_b + Vt$ is disjoint in Lebesgue measure.

Now, to prove that each $R_a$ has a boundary of zero Lebesgue measure, it suffices to use that some $R_{a_0}$ has non-empty interior, and to iterate
\[
	R_{a_0} = \bigcup_{b \xrightarrow{t_n} \dots \xrightarrow{t_0} a_0} N^{n+1} R_b + \sum_{k=0}^n VM^k t_k,
\]
up to have a term of the union of the form $N^{n+1} R_a + t$ completely included in the interior of $R_{a_0}$.
As the union is disjoint in Lebesgue measure, it gives that the boundary of $R_a$ has zero Lebesgue measure.

It finishes the proof of Properties~\ref{props:usual:rauzy}.

\section{Conjugacy with a domain exchange} \label{sec:de}

In this section, we prove Theorem~\ref{thm:de}.
The domain exchange is obtained as a usual Rauzy fractal for a proper substitution thanks to the following.
The following proposition will also be useful to construct finite extensions of torus translations.
It is a generalization of Lemma~8.2.7 in~\cite{AM}.

\begin{prop} \label{props:de}
	Assume Hypothesis~\ref{hyp:rauzy}, and assume that $\sigma$ is proper.
	Then we have the following.
	\begin{itemize}
		\item the unions $R = \bigcup_{a \in A} R_a = \bigcup_{a \in A} R_a + Ve_a$ are disjoint in Lebesgue measure,
		\item we can define a domain exchange almost everywhere by
			\[
				\begin{array}{rcl}
					E : R &\to& R \\
						x &\mapsto& x + Ve_a \text{ if } x \in R_a,
				\end{array}
			\]
			and it is invertible.
		\item $\phi$ is a measurable conjugacy between the uniquely ergodic subshift $(\Omega_\sigma, S)$ and the domain exchange $(R,E, \lambda)$.
	\end{itemize}
\end{prop}

In particular, we have the following theorem.

\begin{thm} \label{thm:proper:de}
	Let $\sigma$ be a primitive proper substitution such that the Perron eigenvalue of the incidence matrix is a unit Pisot number of degree $d+1$.
	Then, the uniquely ergodic subshift $(\Omega_\sigma, S)$ is measurably isomorphic to a domain exchange $(R,E,\lambda)$, with $R \subset \R^{d}$.
\end{thm}

In these results, the hypothesis that $\sigma$ is proper can be replaced with the strong coincidence hypothesis (see~\cite{AI}).
Notice that this result is already stated in~\cite{BV}, Theorem~6, but without proof, and it is proven but not stated in~\cite{SW} (they assume additional hypothesis that are not really used in their proof).
The result could be generalized by avoiding the hypothesis that the Pisot number is a unit by considering $p$-adic spaces, but it would complicate the proof.

\begin{proof}[Proof of Proposition~\ref{props:de}]
	Let $a_0 \in A$ be the letter such that for every $b \in A$, $\sigma(b)$ starts with the letter $a_0$.
	Hence for every letter $b \in A$, $b \xrightarrow{0} a_0$ is a transition in the abelianized prefix automaton.
	Thus, the union $\bigcup_{b \in A} N R_b$ appears in the union $R_{a_0} = \bigcup_{b \xrightarrow{t} a_0} N R_b + Vt$, so it is Lebesgue disjoint by Properties~\ref{props:usual:rauzy}.
	Then,
	\[
		\lambda(R) = \lambda(\bigcup_{b \in A} R_b + Ve_b) \leq \sum_{b \in A} \lambda(R_b) = \lambda(\bigcup_{b \in A} R_b) = \lambda(R),
	\]
	thus the union $\bigcup_{b \in A} R_b + Ve_b$ is also Lebesgue disjoint.
	
	Then, the domain exchange $E$ can be defined almost everywhere and is invertible.
	Then, let $F_0 = \left( \bigcup_{a \in A} \interior(R_a) \right) \cap \left( \bigcup_{a \in A} \interior(R_a) + V e_a \right)$.
	The maps $E$ and $E^{-1}$ are everywhere defined in $F_0$.
	Then, for every $n \in \N$, we define by induction the open subsets $F_{n+1} = E(F_n) \cap E^{-1}(F_n) \cap F_0$.
	The intersection $F = \bigcap_{n \in \N} F_n$ is a subset of $R$ of full Lebesgue measure being invariant under $E$ and $E^{-1}$.
	
	Now we define the natural coding. Since $V$ is one-to-one on $W(u)$, $\phi$ is one-to-one on $\ane(u)$, and we can define the map $\chi : G \to A$, where $G = F \cup \phi(\ane(u))$, by $\chi(x) = a$ if $x \in F \cap R_a$ and $\chi(x) = u_n$ if $x = \phi(S^n u)$. And we can also define $E$ on $\phi(\ane(u))$ by $E(\phi(S^n u)) = \phi(S^{n+1} u) = \phi(S^n u) + V\ab(u_n)$ and $E$ is well-defined on $G$.
	Then, we define the coding map, well-defined on $G$ by
	\[
		\cod : \begin{array}{rcl}
					G &\to& A^\Z \\
					x &\mapsto& (\chi(E^n x))_{n \in \Z}
				\end{array}.
	\]
	We have $\cod \circ E = S \circ \cod$, and the restriction of $\cod \circ \phi$ to $\ane(u)$ is the identity.
	
	Let us show that $\im(\cod) \subseteq \Omega$.
	For $x \in \phi(\ane(u))$, we have $\cod(x) \in \ane(u) \subseteq \Omega$.
	Let $x \in F$.
	For every $N \in \N$, since the set $F_N$ is open
	there exists a neighborhood $U$ of $x$ such that for every $n \in [-N, N]$, $\chi \circ E^n \restr{U}$ is a constant.
	And since $\phi(\ane(u))$ is dense in $R$, $U$ contains an element of $\phi(\ane(u))$.
	Thus, $\cod(x)$ is arbitrarily close to an element of $\ane(u)$ so it is in $\Omega$.
	
	Moreover, the map $\phi$ is continuous.
	Thus for every $x \in F$ and every $\epsilon > 0$, there exists a neighborhood $U$ of $x$ whose image by $\phi \circ \cod$ has a diameter at most $\epsilon$. And for $y \in U \cap \phi(\ane(u)) \cap B(x,\epsilon)$, we have $\phi \circ \cod (y) = y$, thus $\abs{\phi \circ \cod (x) - x} \leq \abs{\phi \circ \cod (x) - \phi \circ \cod(y)} + \abs{y - x} \leq 2 \epsilon$. We deduce that $\phi \circ \cod$ is the identity map of $G$.
	
	Now, let $\mu$ be the push-forward measure of the Lebesgue measure $\lambda$ by the continuous map $\cod \restr{F}$.
	Then $\mu$ is an invariant measure and
	we get that $(\Omega, S, \mu)$ is isomorphic to $(R, E, \lambda)$.
\end{proof}

Thank to Theorem~\ref{thm:dp}, we can proprify $\sigma$.
Thus, Theorem~\ref{thm:de} is a consequence of Theorem~\ref{thm:proper:de}.

\section{Eigenvalues of the dynamical system} \label{sec:vp}

This section aims to prove Theorem~\ref{thm:vp}.
Note that for one implication, we don't need properness or unimodularity:

\begin{prop} \label{prop:vp}
	Let $\sigma$ be a primitive aperiodic substitution.
	Assume that there exists a row vector $w \in \Z^A$ such that for every generalized eigenvector $v$ for an eigenvalue of modulus $\geq 1$, we have
	Then $e^{2 i \pi \alpha}$ is an eigenvalue of $(\Omega_\sigma, S)$.
\end{prop}

For the reciprocal, we need the following proposition (see Proposition~13 in~\cite{durand-petite}).

\begin{prop} \label{prop:dp}
	Let $\sigma$ be a primitive proper substitution.
	If $e^{2 i \pi \alpha}$ is an eigenvalue of $(\Omega_\sigma, S)$ then
	$
		\alpha (1, \dots, 1) M_\sigma^n \xrightarrow[n \to \infty]{} 0 \text{ mod } \Z
	$.
\end{prop}

And we give a characterization of this condition:

\begin{lemme} \label{lem:eqw}
	Let $\alpha \in \R$ and let $M \in M_d(\Z)$ be a pseudo-unimodular matrix.
	We have the equivalence.
	\begin{alignat*}{2}
		 & \alpha (1, \dots, 1) M^n \xrightarrow[n \to \infty]{} 0 \text{ mod } \Z^d, & & \\
		\ArrowBetweenLines
		 & \exists w \in \Z^d, (\alpha (1, \dots, 1) - w)M^n \xrightarrow[n \to \infty]{} 0. & &
	\end{alignat*}
\end{lemme}

To prove this equivalence, we need the following.

\begin{lemme} \label{ml}
	Let $M \in M_d(\Z)$ be a pseudo-unimodular matrix.
	Then there exists $m \in \N_{\geq 1}$ such that for all $y \in \Z^d \cap \im(M^m)$,
	there exists $x \in \Z^d \cap \im(M^m)$ such that $M^mx = y$.
\end{lemme}

\begin{proof}
	Let $m \geq 1$ such that $\ker(M^m)$ and $\im(M^m)$ are supplementary subspaces.
	Since $\im(M^m)$ is a rational subspace, the intersection $\Lambda = \im(M^m) \cap \Z^d$ is a lattice of $\im(M^m)$.
	Let $f : \im(M^m) \to \im(M^m)$ be the restriction of $M^m$ to $\im(M^m)$.
	We have $f(\Lambda) \subseteq \Lambda$, and the pseudo-unimodular hypothesis gives $\det(f) = \pm 1$.
	The matrix of $f$ in a basis of the lattice is in $GL_r(\Z)$, where $r$ is the rank of $M^m$, thus $f^{-1}(\Lambda) \subseteq \Lambda$.
\end{proof}

\begin{proof}[Proof of Lemma~\ref{lem:eqw}]
	Assume that
	\[
		\alpha (1, \dots, 1) M^n \xrightarrow[n \to \infty]{} 0 \text{ mod } \Z^d.
	\]
	Let $a_n \in \Z^d$ be the row vector such that $\alpha (1,...,1)M^n - a_n \in (-1/2, 1/2]^d$.
	We have $a_{n} M - a_{n+1} \xrightarrow[n \to \infty]{} 0$, so there exists $n_0 \in \N$ such that for every $n \geq n_0$,
	$a_{n+1} = a_n M$.
	Let $m \geq 1$ be given by Lemma~\ref{ml} for the matrix $M^t$.
	Then, there exists a row vector $w \in \Z^d$ such that $a_n = w M^n$ for every $n \geq (n_0 +1) m$.
	Hence, we have
	$(\alpha (1,...,1) - w) M^n \xrightarrow[n \to \infty]{} 0$.
	The reciprocal is obvious.
\end{proof}

Now we give another characterization of the condition.

\begin{lemme} \label{lem:charw2}
	Let $\alpha \in \R$, let $M$ be a matrix of size $d$, and let $w \in \Z^d$ be a row vector.
	We have $(\alpha (1, \dots, 1) - w)M^n \xrightarrow[n \to \infty]{} 0$ if and only if
	for every generalized eigenvector $v$ for an eigenvalue of modulus $\geq 1$,
	\[
		\left\{\begin{array}{ll}
			w v = \alpha & \text{ if $v$ has sum $1$}, \\
			w v = 0		& \text{ if $v$ has sum $0$}.
		\end{array}\right.
	\]
	Moreover, the convergence is exponential.
\end{lemme}

\begin{proof}[Proof]
	[$\Longrightarrow$]
	Let us show that for every generalized eigenvector $v$ for an eigenvalue $\beta$ with $\abs{\beta} \geq 1$, we have
	$(\alpha (1,...,1) - w) v = 0$.
	We show it by induction on $k \geq 1$ such that $(M - \beta I)^k v = 0$ and $(M - \beta I)^{k-1} v \neq 0$.
	\begin{itemize}
		\item If $k=1$, then $(\alpha (1,...,1) - w) M^n v = \beta^n (\alpha (1,...,1) - w) v \xrightarrow[n \to \infty]{} 0$.
				Thus, it implies that $(\alpha (1,...,1) - w) v = 0$.
		\item If $k > 1$,
			for every $n \in \N$, we have
			\[
				M^n v - \beta^n v = \sum_{i=1}^{k-1} \binom{n}{i} \beta^{n-i} (M-\beta I)^i v \in \ker((M-\beta I)^{k-1}),
			\]
			so by the induction hypothesis,
			we have $\beta^n (\alpha (1,...,1) - w) v \xrightarrow[n \to \infty]{} 0$ and we conclude.
	\end{itemize}
	
	Now, if $v$ has sum $1$, the equality $(\alpha (1,...,1) - w) v = 0$ implies $\alpha = w v$.
	If $v$ has sum $0$, it implies $w v = 0$.
	
	[$\Longleftarrow$]
	Let $x \in \R^d$. Then, there exists coefficients $c_1, ..., c_{d} \in \C$ and generalized eigenvectors $v_1$, \dots, $v_{d}$ such that $x = c_1 v_1 + \dots + c_{d} v_{d}$. We can assume that each vector $v_i$ as a sum $0$ or $1$.
	Then, we have
	$(\alpha (1,...,1) - w) M^n x = \sum_{i=1}^{d} c_i (\alpha (1,...,1) - w) M^n v_i$.
	If $v_i$ is associated to an eigenvalue $\beta_i$ with $\abs{\beta_i} < 1$, then $M^n v_i \xrightarrow[n \to \infty]{} 0$.
	Otherwise, the hypothesis gives $(\alpha (1,...,1) - w) v_i = 0$.
	Thus, $\forall x \in \R^d$, $(\alpha (1,...,1) - w) M^n x \xrightarrow[n \to \infty]{} 0$.
	We conclude that $(\alpha (1,...,1) - w) M^n \xrightarrow[n \to \infty]{} 0$.
	Moreover, the convergence is exponential.
\end{proof}

\begin{lemme}
	Let $\sigma$ be a primitive aperiodic substitution, and let $\alpha \in \R$.
	Suppose there exists a row vector $w \in \Z^A$ such that
	$(\alpha (1, \dots, 1) - w)M_\sigma^n \xrightarrow[n \to \infty]{} 0$.
	
	Then, $e^{2 i \pi \alpha}$ is an eigenvalue of the dynamical system $(\Omega_\sigma, S)$.
\end{lemme}

\begin{proof}
	Let $v = \alpha (1,...,1) - w$.
	Thanks to Lemma~\ref{lem:charw2}, the convergence of $v M^n$ is exponential.
	Now by Lemma~\ref{lem:cv}, the map $\phi_v = \phi_{u, v, \Omega_\sigma}$ is well-defined, for an admissible fixed point $u$ of $\sigma$ (we can assume that $\sigma$ has an admissible fixed point up to replace $\sigma$ by a power of itself).
	For every $a \in A$, we have $v e_a = \alpha$ mod $\Z$,
	so we have $\phi_v \circ S = \phi_v + \alpha$ mod $\Z$, by Proposition~\ref{props-rauzy-fractal}.

	Thus we get that $\pi \circ \phi_v : \Omega_\sigma \to \R/\Z$ is well defined, where $\pi: \R \to \R/\Z$ is the canonical projection, and we have
	$\pi \circ \phi_v \circ S = \pi \circ \phi_v + \alpha$.
	We conclude that $e^{2i\pi\alpha}$ is an eigenvalue of $(\Omega_\sigma, S)$ for the continuous eigenfunction $e^{2i\pi (\pi \circ \phi_v)}$.
\end{proof}

Now, Proposition~\ref{prop:vp} and Theorem~\ref{thm:vp} are obvious consequences of these lemmas.

In the particular case of unimodular substitutions, Proposition~\ref{prop:dp} gives the following.

\begin{lemme} \label{lem:unimod}
    Let $\sigma$ be a primitive proper unimodular substitution.
    Then, the only eigenvalue of the subshift $(\Omega_\sigma, S)$ being a root of unity is $1$.
\end{lemme}

\begin{proof}
    Let $\alpha \in \Q$ such that $e^{i 2 \pi \alpha}$ is an eigenvalue of the subshift $(\Omega_\sigma, S)$,
    and let $p \geq 1$ be an integer such that $\alpha p \in \Z$. 
    Since $M_\sigma$ is unimodular, it is in the finite group $GL(\Z/p\Z)$ modulo $p$, thus there exists $k \geq 1$ such that $M_\sigma^k$ is the identity matrix modulo $p$.
    By Proposition~\ref{prop:dp} we have $\alpha (1, \dots, 1) M_\sigma^{kn} \xrightarrow[n \to \infty]{} 0 \text{ mod } \Z^A$.
    Thus $\alpha (1,...,1) \xrightarrow[n \to \infty]{} 0 \text{ mod } \Z^A$, so $\alpha \in \Z$.
\end{proof}

\section{Explicit computation of eigenvalues} \label{sec:cvp}

Thanks to Theorem~\ref{thm:dp} and Theorem~\ref{thm:vp}, we can compute the eigenvalues of the subshift for any primitive pseudo-unimodular substitution.
The aim of this section is to provide an explicit computation algorithm.
From an input substitution, we compute a proper substitution in Subsection~\ref{ss:prop}, and then we compute eigenvalues of the subshift from the proper substitution in Subsection~\ref{ss:eig}.

\subsection{Proprification algorithm} \label{ss:prop}
    
    In this subsection, we compute a proprification, as defined in Subsection~\ref{sec:subst}.
    We do it by following the proprification algorithm of F. Durand and S. Petite (see~\cite{durand} and see Corollary~9 in~\cite{durand-petite} for more details).
	
	We start with an input primitive substitution $\sigma$.
	The first step is to compute the \defi{return substitution}.
	To do it
	\begin{itemize}
		\item Replace $\sigma$ by a power of itself to ensure that it has an infinite fixed point.
		\item Let $a \in A$ be the first letter of a left-infinite fixed point $u \in A^\N$.
			We call \defi{return word} on letter $a$, a word $w$ such that $wa$ is a subword of $u$,
			and such that $w$ has a unique occurrence of the letter $a$, at first position.
			
			Let $w_0 \in A^*$ be the unique return word such that $u$ starts by $w_0$.
			Start with $\mathcal{S} = \{w_0\}$.
			
		\item Take out an element $w$ from $\mathcal{S}$.
				Decompose $\sigma(w)$ as a product of return words (such decomposition is unique).
				Add to $\mathcal{S}$ every return word $w$ not already seen.
				Continue until $\mathcal{S}$ is empty.
	\end{itemize}
	The number of return words being finite, this terminates and gives a return substitution $\tau$, whose alphabet is the set $\mathcal{R}$ of return words.
	
	Then, we define a substitution $\xi$ over the alphabet $B = \{ (r,p) \mid r \in \mathcal{R}, 1\leq p \leq \abs{r} \}$ by
	\[
		\xi(r,p) = \left\{ \begin{array}{ll}
								\psi((\tau(r))_p) & \text{ if } 1 \leq p < \abs{r} \\
								\psi((\tau(r))_{[\abs{r}, \abs{\tau(r)}]}) & \text{ if } p = \abs{r}
							\end{array} \right.
	\]
	where $\psi:\mathcal{R}^* \to B^*$ is the morphism defined by $\psi(r) = (r,1)(r,2) \dots (r,\abs{r})$.
	We can show that a power of $\xi$ is left-proper.
	And we easily get a proper substitution from this.
	But the fact that a power of $\xi$ is left-proper is enough to apply our results.
	
	\begin{ex}
		Let $\sigma : 1 \mapsto 213, 2 \mapsto 4, 3 \mapsto 5, 4 \mapsto 1, 5 \mapsto 21$. \\
		The cube of $\sigma$ has a left-infinite fixed point
		\[
			1421352142135213142135214213142135213142...
		\]
		Return words on $1$ are $a = 142$, $b = 1352$ and $c = 13$, and we get the return substitution
		$
		\tau : a \mapsto ababc, b \mapsto abacabc, c \mapsto abac
		$
		which is left-proper.
		
		Then, the substitution $\xi$ is
		$
		0 \mapsto 012, 1 \mapsto 3456, 2 \mapsto 012345678, 3 \mapsto 012, 4 \mapsto 3456, 5 \mapsto 012, 6 \mapsto 78012345678, 7 \mapsto 012, 8 \mapsto 345601278,
		$
		with the identifications $0 = (a,1)$, $1 = (a,2)$, $2 = (a,3)$, $3 = (b,1)$, $4 = (b,2)$, $5 = (b,3)$, $6 = (b,4)$, $7 = (c,1)$, $8 = (c,2)$.
		The square of $\xi$ is left-proper, and its subshift is conjugate to the subshift of $\sigma$.
	\end{ex}
	
\subsection{Computation of eigenvalues for a proper substitution} \label{ss:eig}

Now, a power of $\sigma$ is assumed to be left-proper.
The computation of eigenvalues of the subshift is as follows.

Let $\K$ be the splitting field of the characteristic polynomial of the incidence matrix $M$.
Let $v_0, \dots, v_k$ be a family of vectors of $\K^A$ formed as a concatenation of bases of generalized eigenspaces for eigenvalues of modulus $\geq 1$, and with $v_0$ the Perron eigenvector of sum $1$.
Then, we compute a set $\mathcal{S}$ as follow.
Start with $\mathcal{S} = \emptyset$.
Then, for every $i \in \{1, \dots, k\}$,
\begin{itemize}
	\item if $v_i$ has sum zero, then add it to $\mathcal{S}$,
	\item otherwise normalize $v_i$ such that it has sum $1$, then add $v_i - v_0$ to $\mathcal{S}$.
\end{itemize}
	
	Now, the possible row vectors $w \in \Z^A$ of Theorem~\ref{thm:vp} are exactly those that are orthogonal to every vector of $\mathcal{S}$. We can describe this set as the kernel of an integer matrix by the following.
	
Choose a basis of the number field $\K$ seen as a $\Q$-vector space.
Decompose each vector of $\mathcal{S}$ in this basis.
Form a matrix $M_{\mathcal{S}}$ whose columns are these components.
Multiply in place the matrix $M_{\mathcal{S}}$ by an integer in order to have $M_{\mathcal{S}}$ with coefficients in $\Z$.
	
Then, we obtain eigenvalues as the set
\[
	\{ e^{2 i \pi w v_0} \mid w \in \Z^A \text{ and } w M_{\mathcal{S}} = 0\}.
\]
And we can completely describe this set by computing a basis of the $\Z$-module of such $w$, and then computing a basis of the $\Z$-module of possible $w v_0$, using for example the Schmidt normal form.

\begin{ex}[Example~2 of \cite{FMN}]
	Let $\sigma : a \mapsto abdd, b \mapsto bc, c \mapsto d, d \mapsto a$.
	Its incidence matrix 
	\[
		\left(\begin{array}{rrrr}
			1 & 0 & 0 & 1 \\
			1 & 1 & 0 & 0 \\
			0 & 1 & 0 & 0 \\
			2 & 0 & 1 & 0
		\end{array}\right)
	\]
	is irreducible and has two eigenvalues of modulus $\geq 1$, associated to eigenvectors of non-zero sum.
	Hence, the set $\mathcal{S}$ is reduced to one element which is the difference of the two different eigenvectors of sum $1$.
	Here, both eigenvectors live in the field is $\Q(\beta)$, with $\beta$ the Perron eigenvalue, so we can do the computation in this field rather than the splitting field.
	The computation in this field gives $\mathcal{S} = \{ \frac{1}{41} (-6 \beta^{3} + 9 \beta^{2} - 17 \beta + 7,\,18 \beta^{3} - 27 \beta^{2} - 31 \beta + 20,\,-14 \beta^{3} + 21 \beta^{2} + 15 \beta - 11,\\
	2 \beta^{3} - 3 \beta^{2} + 33 \beta - 16)^t \}$. Thus, in the basis $(1, \beta, \beta^2, \beta^3)$, the matrix $M_\mathcal{S}$ is
	\[
		M_\mathcal{S} = \left(\begin{array}{rrrr}
							7 & -17 & 9 & -6 \\
							20 & -31 & -27 & 18 \\
							-11 & 15 & 21 & -14 \\
							-16 & 33 & -3 & 2
						\end{array}\right).
	\]
	Then, the set of row vectors $w \in \Z^4$ such that $w M_\mathcal{S} = 0$ is the $\Z$-module generated by $(1,1,1,1)$ and $(0,3,4,1)$.
	Then, we get
	\[
		\{ w v_0 \mid w \in \Z^A \text{ and } w M_{\mathcal{S}} = 0\} = \Z + (-\beta^2 + \beta + 2) \Z = \Z[\sqrt{2}],
	\]
	where $v_0$ is the Perron eigenvector of sum $1$.
	
	We can moreover check that $\sigma$ has no non-trivial coboundary and it has an initials period of $1$ (see~\cite{host} for more details), so the eigenvalues of the subshift are indeed $e^{2 i \pi n \sqrt{2}}$, $n \in \Z$.
	The computation from a proprification leads to the same result.
	Notice that there is a mistake in~\cite{FMN} Example~2. They claim that eigenvalues are $e^{i \pi n \sqrt{2}}$, $n \in \Z$, but their computation leads to the same result as us, they just forgot a $2$ in their conclusion.
\end{ex}

We provide an implementation of this algorithm in the Sage computing system (see~\url{https://www.sagemath.org/}):

\noindent\url{http://www.i2m.univ-amu.fr/perso/paul.mercat/Eigenvalues.html}

\section{Finite extension of a torus translation} \label{sec:ef}

This section mainly aims to prove Theorem~\ref{thm:ef}.
But before we show that the hypothesis that the Perron eigenvalue is Pisot is necessary.

\begin{prop} \label{prop:pisot}
	Let $\sigma$ be a primitive pseudo-unimodular substitution.
	If the Perron eigenvalue $\beta$ of the incidence matrix is not Pisot, then the subshift has strictly less rationally independent eigenvalues than the degree of $\beta$.
\end{prop}

\begin{proof}
	Thanks to Theorem~\ref{thm:dp}, we can assume that $\sigma$ is proper.
	Let $\beta > 1$ be the Perron eigenvalue of the incidence matrix.
	Let $\gamma$ be a conjugate of $\beta$ with $\abs{\gamma} \geq 1$.
	Let $\varphi : \Q(\beta) \to \Q(\gamma)$ be the morphism of fields such that $\varphi(\beta) = \gamma$.
	Let $e^{2 i \pi \alpha}$ be an eigenvalue of the subshift.
	Let $v_\beta$ be the Perron eigenvector of sum $1$.
	Then, $v_\gamma = \varphi(v_\beta)$ is an eigenvector of sum $1$ for the eigenvalue $\gamma$.
	By Theorem~\ref{thm:vp}, there exists a row vector $w \in \Z^A$ such that $\alpha = w v_\beta = w v_\gamma$.
	Thus, we have $\varphi(\alpha) = \alpha$.
	Hence, $\alpha$ lives in the $\Q$-vector space $\{x \in \Q(\beta) \mid \varphi(x) = x\}$ whose dimension is strictly less than $\deg(\beta)$ since it doesn't contains $\beta$.
\end{proof}

\subsection{Proof of Theorem~\ref{thm:ef}}

    Our proof of Theorem~\ref{thm:ef} is similar to the proof of the main theorem in~\cite{durand-petite}.
    The idea is the following.
    We start by showing that there exists rationally independent eigenvalues.
    Then, we get a minimal translation on the torus $\T^d$ as a factor, with a Rauzy fractal $R'$.
    Then, we use the fact that the substitution is proper to construct a domain exchange on another Rauzy fractal $R$, which is conjugate to the subshift.
    And we define a map $\psi : R \to R'$ being a translation by pieces, thus finite-to-one, so the extension is finite.
    
	\tikzcdset{row sep/normal=1.5cm}
	\tikzcdset{column sep/normal=1.5cm}
	\begin{center}
		\begin{tikzcd}
			\Omega_\xi \arrow[out=120,in=60,loop,"S"] \arrow[r,"\sim", hook, two heads] & \Omega_\sigma \arrow[out=120,in=60,loop,"S"] \arrow[rd, "\phi_{V'}", two heads] \arrow[r,"\phi_V", hook, two heads] & R \arrow[out=120,in=60,loop,"E"] \arrow[d, "\psi", dashed, two heads] &\\
			& & R' \arrow[r, "\pi", two heads] & \T^d \arrow[out=120,in=60,loop,"T_\alpha"]
		\end{tikzcd}
	\end{center}

    \subsubsection{Existence of rationally independent eigenvalues}
	Let us show that there exists rationally independent numbers $1$, $\alpha_1$, ..., $\alpha_d$
	such that $e^{2 i \pi \alpha_1 }$, ..., $e^{2 i \pi \alpha_d}$ are eigenvalues of the subshift.
	We decompose the minimal polynomial of $M = M_\sigma$ in the form $\pi_\beta X^k Q$, where $\pi_\beta$ is the minimal polynomial of the Perron eigenvalue $\beta$, and with $Q(0) \neq 0$.
	We can assume that $k=1$ up to replace $\sigma$ by $\sigma^k$.
	Then, the row vector $(1, \dots, 1)$ is orthogonal to $\ker(Q(M))$ thank to the following lemma.
	
	\begin{lemme} \label{lem:conj}
		Let $M$ be a matrix with integer coefficients.
		If $\gamma$ is an eigenvalue of $M$ such that every associated generalized eigenvector has sum zero, then every generalized eigenvector associated with a conjugate of $\gamma$ also has sum zero.
	\end{lemme}
	
	\begin{proof}
		Let $\beta$ be a conjugate of $\gamma$, and let $\varphi : \Q(\beta) \to \Q(\gamma)$ be the morphism of fields sending $\beta$ to $\gamma$.
		Let $v$ be a generalized eigenvector for the eigenvalue $\beta$.
		Let $k \geq 1$ such that $(M - \beta I)^k v = 0$.
		Then it gives $(M - \gamma I)^k \varphi(v) = 0$.
		Thus, $\varphi(v)$ is a generalized eigenvector for the eigenvalue $\gamma$.
		And, we have $\varphi((1,...,1) v) = (1,...,1) \varphi(v) = 0$, so $(1,...,1) v = 0$.
	\end{proof}
	
	Let $w_0 = (1, \dots, 1)$, $w_1$, \dots, $w_d$ be row vectors in the orthogonal complement of $\ker(Q(M))$ that are linearly independent modulo the orthogonal complement of $\ker((\pi_\beta Q)(M))$. As we have rational subspaces, we can assume that every vector $w_i$ is in $\Z^A$.
	
	Let $v_0$ be the Perron eigenvector of sum $1$. Let us show that $1$, $w_1 v_0$, \dots, $w_d v_0$ are rationally independent.
	Let $c_0$, \dots, $c_d \in \Q$ such that $c_0 w_0 v_0 + \dots + c_d w_d v_0 = 0 = (c_0 w_0 + \dots + c_d w_d) v_0$.
	Then, as $c_0 w_0 + \dots + c_d w_d$ is rational, it is orthogonal to every conjugate of $v_0$, thus to $\ker(\pi_\beta(M))$.
	As $c_0 w_0 + \dots + c_d w_d$ is orthogonal to $\ker((\pi_\beta Q)(M))$, every $c_i$ is zero.
	And by Proposition~\ref{prop:vp}, $e^{2i \pi \alpha_i}$ are eigenvalues of the dynamical system, where $\alpha_i = w_i v_0$, since every generalized eigenvector of zero sum, for an eigenvalue of modulus at least $1$, is in $\ker(Q(M))$.
	
	\subsubsection{Minimal torus translation as a factor}
	
	Now, we show that the subshift is an extension of a minimal translation of the torus $\T^d$.
	Let $V'$ be the matrix whose rows are $\alpha_i (1,...,1) - w_i$, $i = 1,...,d$.
	By Lemma~\ref{lem:charw2} the convergence $V' M^n \xrightarrow[n \to \infty]{} 0$ is exponential.
	Thus, by Lemma~\ref{lem:cv}, the map $\phi_{V'} = \phi_{u, V', \Omega_\sigma}$ of Proposition~\ref{prop-rauzy-fractal} is well-defined. It permits to define a Rauzy fractal $R' = \phi_{V'}(\Omega_\sigma)$.
	And for every $x \in \Omega_\sigma$ we have $\phi_{V'} (S x) = \phi_{V'}(x) + V' \ab(x_0) = \phi_{V'}(x) + \alpha$ mod $\Z^A$, where $\alpha = (\alpha_1, ..., \alpha_d) \in \R^d$.
	Thus, by the map $\pi \circ \phi_{V'} : \Omega \to \T^d$, the subshift is an measurable extension of the translation by $\alpha$ on the torus $\T^d$.
	And this translation is minimal since $1$, $\alpha_1$, ..., $\alpha_d$ are linearly independent over $\Q$.

	\noindent To end the proof of Theorem~\ref{thm:ef}, it remains to show that this extension is finite.
	
	\subsubsection{Construction of a domain exchange}
	
	Now, we show that the subshift is conjugate to a domain exchange.
	Let $\Pi$ be the projection on $\im(M) = \ker((\pi_\beta Q)(M))$ along $\ker(M)$. We have $\Pi M = M$, and we can use the following lemma.
	
	\begin{lemme}
		Let $\Pi$ be a rational matrix such that $\Pi M = M$.
		Then, the row vectors $(1,...,1)\Pi$, $w_1 \Pi$, ..., $w_d \Pi$ are linearly independent.
	\end{lemme}
	
	\begin{proof}
		Let $c_0$, $c_1$, ..., $c_d$ be such that $c_0 (1,...,1) \Pi + c_1 w_1 \Pi + ... + c_d w_d \Pi = 0$.
		We can assume that every $c_i$ is in $\Q$ since $\Pi$ and $w_i$ have rational coordinates.
		Then, $w_i M^n = \alpha_i(1,...,1) M^n + o_{n \to \infty}(1)$,
		so
		\[
		    (c_0 + c_1 \alpha_1 + ... + c_d \alpha_d) (1,...,1) M^n = o_{n \to \infty}(1).
		\]
		As, $(1,...,1)M^n$ diverges when $n$ tends to infinity, we get $c_0 + c_1 \alpha_1 + ... + c_d \alpha_d = 0$.
		Then, the linear independence over $\Q$ gives $c_i=0$ for every $i$.
	\end{proof}
	
	This lemma tells us that $V = V' \Pi$ is of rank $d$.
	Now, let us show that $V$ satisfies Hypothesis~\ref{hyp:rauzy}, in order to have a usual Rauzy fractal.
	
	We want to show that $\ker((M-\beta I)MQ(M)) \subseteq \ker(V)$.
	As $\rk(V) = d = \codim \ker((M-\beta I)MQ(M))$,
	it will give also the other inclusion.
	Since polynomials $X-\beta$, $X$ and $Q$ are pairwise coprime, we have
	\[
	    \ker((M-\beta I)MQ(M)) = \ker(M-\beta I) \oplus \ker(M) \oplus \ker(Q(M)).
	\]
	By definition of $\Pi$, we have $\ker(M) = \ker(\Pi) \subset \ker(V)$.
	By construction, each $w_i$ is orthogonal to $\ker(Q(M))$, and the row vector $(1, \dots, 1)$ is also orthogonal to $\ker(Q(M))$, thus we have $\ker(Q(M)) \subseteq \ker(V')$.
	And we also have $\ker(M-\beta I) \subseteq \ker(V')$ by Lemma~\ref{lem:charw2}.
	Since, $\Pi M = M \Pi$, we obtain $\ker(M-\beta I) \subset \ker \Pi (M-\beta I) = \ker (M - \beta I) \Pi \subseteq \ker(V)$, and similarly $\ker Q(M) \subset \ker(V)$.
	Then, we get the equality $\ker(V) = \ker((M-\beta I)M Q(M))$.
	
	Hence, Hypothesis~\ref{hyp:rauzy} is satisfied, and $\phi_{u, V, \Omega_\sigma}$ is well-defined and defines a usual Rauzy fractal, for an admissible fixed point $u$ (which exists up to replace $\sigma$ by a power of itself).
	
	Moreover, since the substitution $\sigma$ is proper, Proposition~\ref{props-rauzy-fractal} gives us a domain exchange
	on $R = \phi_V(\Omega_\sigma)$ measurably conjugated to the subshift.
	
	\subsubsection{Translation by pieces}
	
	Now, we show that the almost everywhere defined map $\psi = \phi_{V'} \circ \phi_V^{-1} : R \to R'$ is a translation by pieces.
	For almost every $x \in \Omega$, we have
	\[
		\psi \left( \sum_{n \in \N} V M^n t_n(x) \right) = \sum_{n \in \N} V' M^n t_n(x).
	\]
	But for every $n \geq 1$ we have $V M^n = V \Pi M^n = V' M^n$.
	Thus $\psi$ is a translation by $(V'-V) t$ on each piece $\phi_V(t_0^{-1}(t)) = \bigcup_{b \xrightarrow{t} a} N R_b + Vt$, for every label $t$ of the abelianized prefix automaton, where $t_0 = \ab \circ p_0 : \Omega_\sigma \to \Z^d$ is the map defined in Subsection~\ref{ss:dt}.
	
	\subsubsection{End of the proof of Theorem~\ref{thm:ef}}
	
	Since $\psi$ is a translation by pieces, it is finite-to-one.
	The map $\pi : \R^d \to \T^d$ restricted to $R'$ is also finite-to-one since $R'$ is bounded.
	Hence $\pi \circ \psi : (R,E, \lambda) \to (\T^d, T_\alpha, \lambda)$ is finite-to-one, where $T_\alpha$ is the translation by $\alpha = (\alpha_1, \dots, \alpha_d)$.
	Thus, the continuous map $\pi \circ \phi_{V'} : (\Omega, S) \to (\T^d, T_\alpha)$ is almost everywhere finite-to-one.
	It ends the proof of Theorem~\ref{thm:ef}.

\begin{rem}
	The map $\psi = \phi_{V'} \circ \phi_{V}^{-1}$ can always be defined almost everywhere, as in this proof, as soon as $V$ and $V'$ are well-defined with $\phi_V$ almost everywhere invertible.
	And we have
	\[
		\psi : \begin{array}{ccc}
					R &\to& R' \\
					\sum_{n=0}^\infty V M^n t_n &\mapsto& \sum_{n=0}^\infty V' M^n t_n,
				\end{array}
	\]
	where $t_n$ are labels of left-infinite paths in the abelianized prefix automaton.
	But in general this map does not seems to be always finite-to-one (see Example~\ref{ex:2pisots2}).
\end{rem}

In the following subsection, we prove the last remaining theorem to prove.

\subsection{Proof of Theorem~\ref{thm:wip}}
	By Theorem~\ref{thm:dp}, we can assume that $\sigma$ is proper.
	And up to replace $\sigma$ by a power of itself, we can also assume that the only eigenvalue of the matrix being a root of unity is $1$.
	
	If there is no generalized eigenvector $v$ of sum $1$ for the eigenvalue $1$, then the hypothesis of Theorem~\ref{thm:ef} is satisfied, and the subshift is a finite extension of a translation on the torus $\T^d$.
	
	Otherwise, we use Theorem~\ref{thm:vp}. Let $v$ be a generalized eigenvector of sum $1$ for the eigenvalue $1$.
	As the generalized eigenspace for the eigenvalue $1$ is rational, we can assume that $v$ has rational coordinates.
	Then the eigenvalues of the subshift are of the form $e^{2i \pi \alpha}$ with $\alpha = w v \in \Q$, for some row vectors $w \in \Z^A$.
	Hence eigenvalues of the subshift are roots of unity, and it ends the proof of Theorem~\ref{thm:wip}.

\section{Examples} \label{sec:ex}

\begin{ex}[Non-Pisot, \cite{FMN}] \label{ex:fmn1}
	For the primitive unimodular substitution	
	\[
		 a \mapsto abbbccccccccccdddddddd,
		 b \mapsto bccc,
		 c \mapsto d,
		 d \mapsto a
	\]
	the subshift has eigenvalues $e^{2 i n \pi \sqrt{2}}$, $n \in \Z$.
	The characteristic polynomial $x^{4} - 2x^{3} - 7x^{2} - 2x + 1$ of $M_\sigma$ is irreducible, and the Perron eigenvalue $\theta_1$ is not Pisot since there is another root $\theta_4$ of modulus $> 1$.
	The subshift is not weakly mixing, contrary to what is said in~\cite{FMN} Example~1.
	They made a miscalculation. They say that $\alpha = \frac{\theta_1^3}{(1+\sqrt{10})\theta_1 + 11 - \sqrt{10}}Q(\frac{1}{\theta_1}) = \frac{\theta_4^3}{(1 - \sqrt{10})\theta_4 + 11 + \sqrt{10}} Q(\frac{1}{\theta_4})$ can take only integer values, where $Q$ is a polynomial over $\Z$, but it is false since for example $Q(X) = 6 X - 11 X^2 + 19 X^3$ is a solution that gives $\alpha = \frac{7 + \sqrt{2}}{3}$.
\end{ex}

\begin{ex}[Conjugate to a torus translation]
	The subshift of the weakly irreducible Pisot substitution
	$
		1 \mapsto 2,
		2 \mapsto 3,
		3 \mapsto 14,
		4 \mapsto 5,
		5 \mapsto 1425
	$
	is measurably conjugate to a translation on $\T^2$.
	The incidence matrix of a left-proprification has only one eigenvalue of modulus $\geq 1$, thus Theorem~\ref{thm:ef} applies.
	The domain exchange and fundamental domain are depicted in Figure~\ref{fig:tribo}.
	The eigenvalues of the dynamical system are $e^{2i \pi \alpha}$, with $\alpha \in \frac{2 \beta^2 + \beta - 6}{11} \Z + (\beta^2 + \beta) \Z$, where $\beta$ is the Pisot number root of $X^3-X^2-X-1$.
\end{ex}

\begin{figure}
	\centering
	\caption{Domain exchange and fundamental domain of $\T^2$ for a subshift being measurably conjugate to a translation of $\T^2$} \label{fig:tribo}
	\includegraphics[scale=.25]{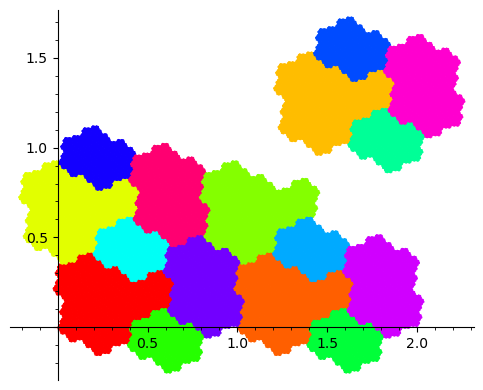} \quad
	\includegraphics[scale=.25]{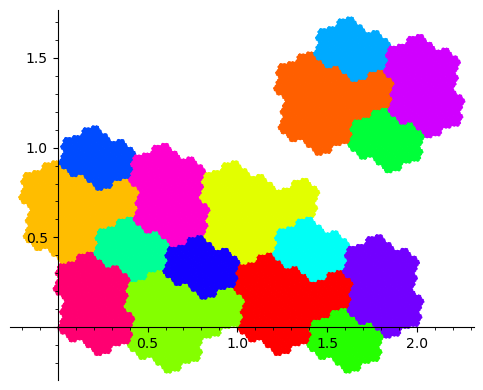} \quad \quad
	\includegraphics[scale=.25]{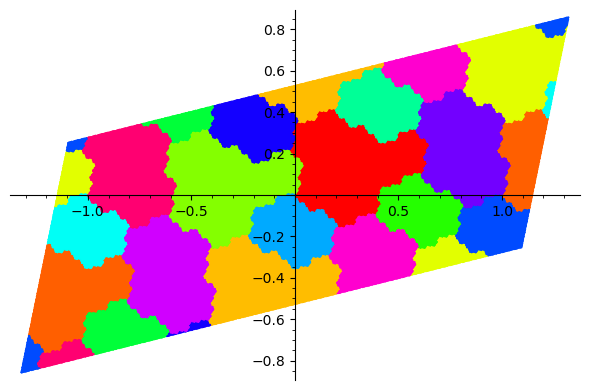}
\end{figure}

\begin{ex}[$2$--to--$1$ extension of a torus translation]
	It can be shown that the subshift of the weakly irreducible Pisot substitution
	$
		\sigma:
		a \mapsto Ab,
		b \mapsto A,
		A \mapsto aB,
		B \mapsto a
	$
	is a $2$--to--$1$ extension of a translation on $\T^1$.
	Note that it is a $2$--to--$1$ cover of the subshift of the original Fibonacci substitution $a \mapsto ab, b \mapsto a$.
	A left-proprification of $\sigma$ has an incidence matrix with eigenvalues zero, roots of unity associated to generalized eigenvectors of sum zero, and golden number and conjugate. Thus Theorem~\ref{thm:ef} applies.
	An approximation of the graph of $\psi$ is plotted in Figure~\ref{fig:psi} (right).
	The eigenvalues of the subshift are $e^{2i\pi n \varphi}$, $n \in \Z$, where $\varphi$ is the golden number.
\end{ex}

\begin{ex}[Presumably infinite extension of a torus translation]
	The subshift of the primitive substitution
	$
		 1 \mapsto 11116,
		 2 \mapsto 1,
		 3 \mapsto 1111112,
		 4 \mapsto 1111113,$
		 $
		 5 \mapsto 466,
		 6 \mapsto 566
 	$
	is a (presumably infinite) extension of a minimal circle translation.
	Its incidence matrix is not diagonalizable.
	Its characteristic polynomial is
	$
		(x^2 - 4x - 1) (x^2 - x - 1)^2.
	$
	The eigenvalues of the dynamical system are $e^{2 i \pi \alpha}$ where $\alpha \in \frac{45}{2} \varphi \Z$, where $\varphi$ is the golden number.
\end{ex}

\begin{ex}[Two Pisot, intermediate]
	The subshift of the primitive substitution
	\[
		 1 \mapsto 1116,
		 2 \mapsto 1,
		 3 \mapsto 2,
		 4 \mapsto 3,
		 5 \mapsto 1146,
		 6 \mapsto 566
	 \]
	is neither a finite extension of a minimal translation of $\T^3$ nor weakly mixing.
	Indeed, the eigenvalues of the subshift are $e^{2 i \pi \alpha}$, where $\alpha \in 57 \varphi \Z$, where $\varphi$ is the golden number, but the degree of the Perron eigenvalue is $4$.
	The characteristic polynomial of this matrix is
	\[
		(x^{2} - x - 1) \cdot (x^{4} - 4x^{3} + 2x^{2} - x + 1).
	\]
\end{ex}

\begin{figure}
	\centering
	\caption{Weakly mixing domain exchange associated to a substitution whose incidence matrix has two Pisot eigenvalues} \label{fig:wm}
	\includegraphics[scale=.3]{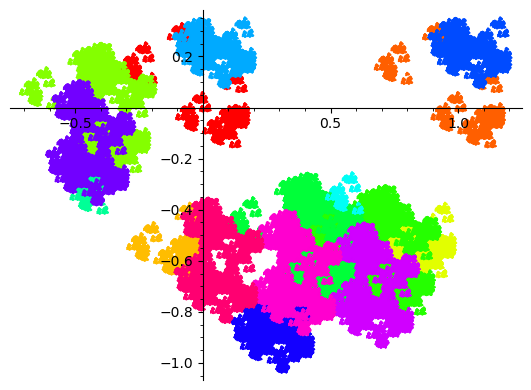} \quad \quad
	\includegraphics[scale=.3]{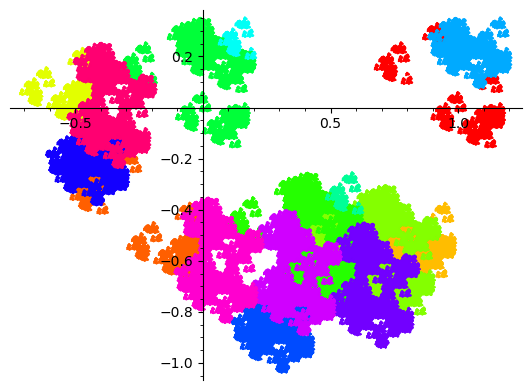}
\end{figure}

\begin{ex}[Two Pisot, weakly mixing]
	The primitive substitution
	\[
		1 \mapsto 15,
		2 \mapsto 2122,
		3 \mapsto 122,
		4 \mapsto 13,
		5 \mapsto 14122
	\]
	gives a weakly mixing subshift.
	Its incidence matrix has two Pisot eigenvalues of degrees $2$ and $3$.
	But we can describe it geometrically with a domain exchange, see Figure~\ref{fig:wm}.
\end{ex}

\begin{ex}[Two Pisot, presumably infinite extension of $(\T^2,T)$]
	The subshift of the substitution
	$
		1 \mapsto 16,
		2 \mapsto 122,
		3 \mapsto 12,
		4 \mapsto 3,
		5 \mapsto 124,
		6 \mapsto 15
	$
	is an (presumably infinite) extension of a translation on the torus $\T^2$.
	Note that the square of this substitution is left-proper.
	The incidence matrix has two Pisot eigenvalues of degrees $3$.
	The eigenvalues of the dynamical system are $e^{2 i \pi \alpha}$, where $\alpha \in 3\beta \Z + 3\beta^3 \Z$, where $\beta$ is the Perron eigenvalue of the incidence matrix.
	We show the domain exchange and the image by $\psi$ in figure~\ref{fig:2pisots:psi}.
\end{ex}

\begin{figure}
	\centering
	\caption{Domain exchange and multi-tiling associated to a substitution whose incidence matrix has two Pisot eigenvalues, and whose subshift is an (presumably infinite) extension of a translation on $\T^2$} \label{fig:2pisots:psi}
	\includegraphics[scale=.25]{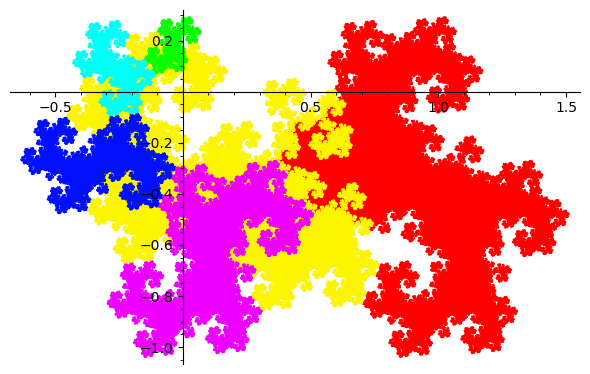}
	\includegraphics[scale=.25]{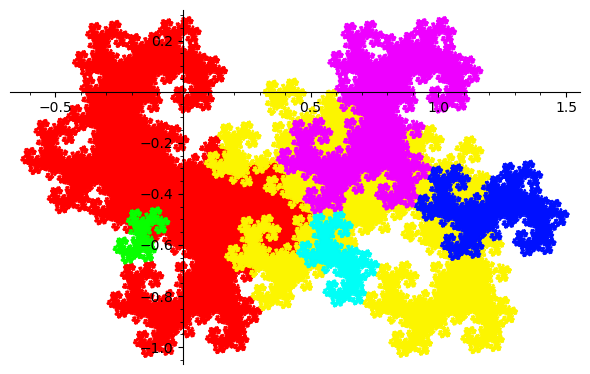}
	\includegraphics[scale=.25]{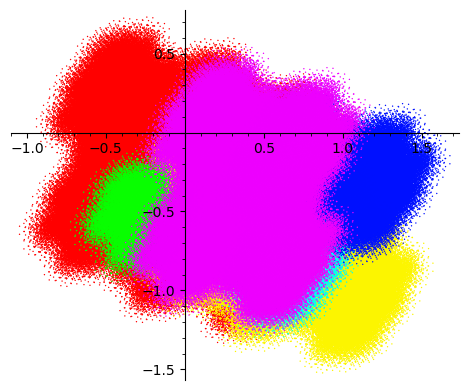}
\end{figure}

\begin{figure}
	\centering
	\caption{Approximation of the graph of the function $\psi$, for an example with two Pisot eigenvalues (left) and for an example weakly irreducible Pisot (right)} \label{fig:psi}
	\includegraphics[scale=.3]{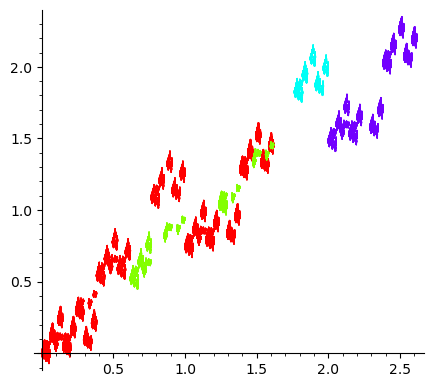} \quad \quad
	\includegraphics[scale=.4]{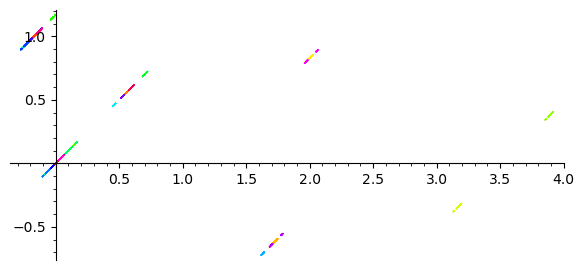}
\end{figure}

\begin{ex}[Two Pisot, presumably infinite extension of $(\T^1,T)$] \label{ex:2pisots2}
	The subshift of the substitution
	$
		1 \mapsto 114,
		2 \mapsto 122,
		3 \mapsto 2,
		4 \mapsto 13
	$
	is an (presumably infinite) extension of a translation on $\R/\Z$.
	The incidence matrix has two Pisot eigenvalues of degrees $2$.
	We plot an approximation of the graph of the function $\psi$ on the Figure~\ref{fig:psi} (left).
	We can see that the almost everywhere defined function $\psi$ seems to be infinite-to-one.
	The eigenvalues of the subshift are $e^{2 i \pi \sqrt{5} n}$ where $n \in \Z$.
\end{ex}

\begin{figure}
	\centering
	\caption{Domain exchange, Rauzy fractal with overlaps, and fundamental domain of $\T^2$ for Example~\ref{ex:timo}} \label{fig:timo}
	\includegraphics[scale=.3]{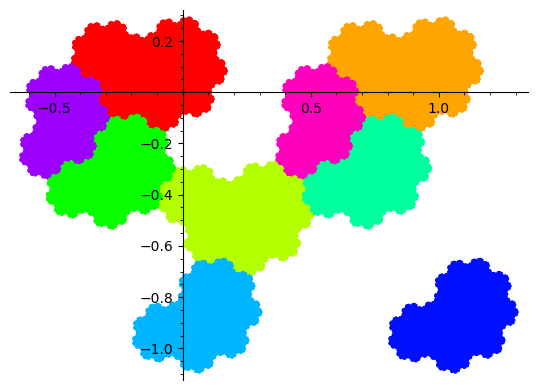} \quad 
	\includegraphics[scale=.3]{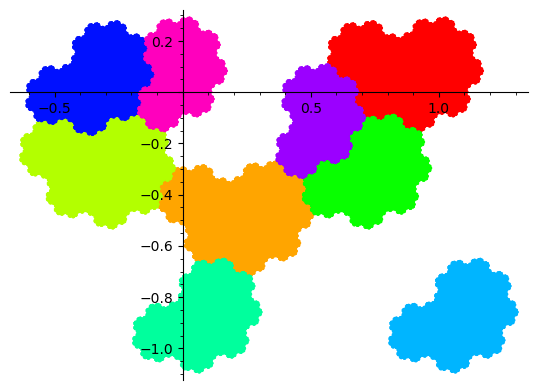} \\
	\includegraphics[scale=.3]{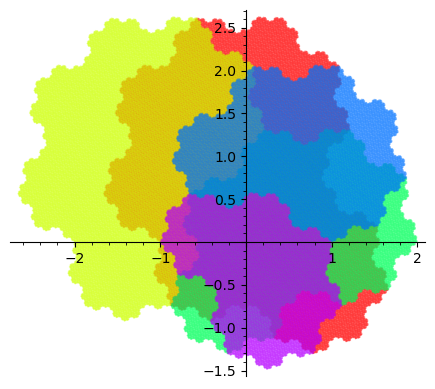} \quad
	\includegraphics[scale=.275]{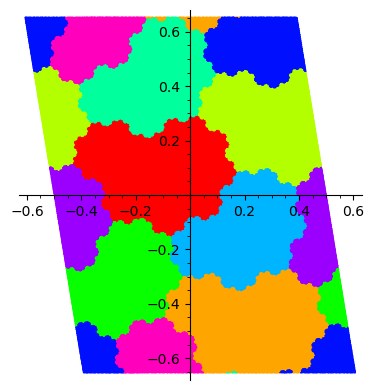} 
\end{figure}

\begin{ex}[due to Timo Jolivet] \label{ex:timo}
	The primitive substitution
	\[
		\sigma:
		1 \mapsto 213,
		2 \mapsto 4,
		3 \mapsto 5,
		4 \mapsto 1,
		5 \mapsto 21
	\]
	is an example for which the usual Rauzy fractal overlaps (see Figure~\ref{fig:timo}).
	However, we can proprify it and it permits to obtain a domain exchange that is measurably conjugate to the subshift of $\sigma$.
	Moreover, the proprification of $\sigma$ satisfies Theorem~\ref{thm:ef}, thus it is also a finite extension of a torus translation.

    Notice that the return substitution of $\sigma^3$ on any letter has only three letters.
	The eigenvalues of the subshift are $e^{2 i \pi \alpha}$, where $\alpha \in \Z[\beta]$, where $\beta$ is the real root of $X^3-2X^2+X-1$.
\end{ex}

\begin{ex}[Two eigenvalues]
    The square of the substitution
    \[
        \sigma:
        \left\{\begin{array}{l}
            0 \mapsto 1203  \\
            1 \mapsto 12 \\
            2 \mapsto 13 \\
            3 \mapsto 03
        \end{array}\right.
    \]
    is proper, primitive and pseudo-unimodular.
    The associated subshift has eigenvalues $\{-1, 1\}$, thus the eigenvalue $1$ of the square $(\Omega_\sigma, S^2)$ is not simple, so the square is not ergodic. By Theorem~IV.1 in~\cite{dekking}, it implies that $(\Omega_\sigma, S^2)$ is not minimal, and this can be indeed easily checked directly.
\end{ex}

\begin{ex}[Family of weakly mixing subshifts]
    For every $n \geq 1$, the substitution
    \[
        \left\{
        \begin{array}l
            a \mapsto ab  \\
            b \mapsto ac^{2n-1} \\
            c \mapsto ac^{2n}
        \end{array}
        \right.
    \]
    is primitive, unimodular and left-proper.
    Its incidence matrix $\begin{pmatrix} 1 & 1 & 1 \\ 1 & 0 & 0 \\ 0 & 2n-1 & 2n \end{pmatrix}$ has eigenvalues $\{1, n \pm \sqrt{n^2+1} \}$, and the eigenvalue $1$ is associated with the eigenvector $\begin{pmatrix} 1 \\ 1 \\ -1\end{pmatrix}$ of sum $1$ with integer coefficients.
    Thus, by Theorem~\ref{thm:vp}, the unique eigenvalue of the subshift is $1$, thus it is weakly mixing.
\end{ex}

\section{Acknowledgements}

I thank Pascal Hubert for his careful reading of this article. And I thank Fabien Durand and Samuel Petite for interesting discussions. I also thank the referee for is careful reading and for all his comments that improved the article.


\end{document}